\documentclass[12pt]{article}
\usepackage{amsmath}
\usepackage{amsthm,amsmath,amsfonts,amssymb,color}
\usepackage{verbatim}
\usepackage{mathrsfs}
\usepackage[alphabetic]{amsrefs}
\usepackage{graphicx}
\numberwithin{equation}{section}

\theoremstyle{plain}

\newtheorem{thm}{Theorem}
\newtheorem{lemma}[thm]{Lemma}
\newtheorem{cor}[thm]{Corollary}
\newtheorem{prop}[thm]{Proposition}
\newtheorem{definition}[thm]{Definition}

\theoremstyle{remark}
\newtheorem{rem}[thm]{Remark}

\newcommand{\R}{\mathbb{R}}
\newcommand{\Rn}{\mathbb{R}^n}

\newcommand{\Lps}{\mathcal{L}_p^s\,}

\newcommand{\dd}{\partial}
\newcommand{\e}{\varepsilon}
\newcommand{\h}{\mathfrak{h}}

\DeclareMathOperator*{\osc}{osc}

\DeclareMathOperator*{\spt}{supp}

\newcommand{\lap}{\Delta}

\newcommand{\ww}{W^{s,p}(\R^n)}
\newcommand{\wwz}{W^{s,p}_0(\Omega)}
\renewcommand{\u}{\mathscr{U}}
\renewcommand{\l}{\mathscr{L}}
\begin{document}
\title{Perron's Method and Wiener's Theorem for a Nonlocal Equation}
\author{Erik Lindgren\footnote{eriklin@kth.se, Department of Mathematics, KTH. }  \qquad Peter Lindqvist\footnote{lqvist@math.ntnu.no, Department of Mathematics, NTNU.}}
\maketitle
\begin{abstract}\noindent 
We study the Dirichlet problem for non-homogeneous equations involving the fractional $p$-Laplacian. We apply Perron's method and prove Wiener's resolutivity theorem.
 \end{abstract}
 
\section{Introduction}
The object of our work is to study the Dirichlet problem
\begin{equation}\label{eq:mainn}
\begin{cases}
(-\lap_p)^s u = f\quad \text{in}\quad & \Omega,\\
u = g\quad \qquad \text{in}\quad & \R^n\setminus \Omega,
\end{cases}
\end{equation}
where $s\in (0,1)$, $p\in (1,\infty)$ and
$$
(-\lap_p)^s u\, (x):=2 \,\text{PV} \int_{\R^n}\frac{|u(x)-u(x+y)|^{p-2}(u(x)-u(x+y))}{|y|^{n+sp}}\, dy.
$$
This is a non-linear counterpart to the fractional Laplace Equation
$$(-\Delta)^s u\, =\, f.$$ 

We have found the Perron method suitable to investigate the boundary behaviour.
The celebrated method, introduced in 1923 by O. Perron (cf. \cite{Per23}) for the Laplace Equation, has a wide range of applications. It requires little, mainly that a comparison principle is valid. It has been adapted to many non-linear equations and appears in the modern theory of viscosity solutions, too.
We shall study the Dirichlet boundary value problem in a bounded domain $\Omega$ in $\R^n$
for an equation that arises in the following way.
The problem of minimizing the variational integral
\begin{equation} \label{ett} J(v)\,=\,\frac{1}{p}\int\limits_{\R^n\,}\!\!\int\limits_{\R^n}\frac{|v(y)-v(x)|^p}{|y-x|^{n+sp}}\,dy\,dx\,-\,\int\limits_{\Omega}\!fv\,dx
\end{equation}
among all functions $v$ with given boundary values $g$ in the sense that $v-g \in W_0^{s,p}(\Omega)$ leads to the Euler-Lagrange Equation
\begin{equation}\label{tvaa}	\int\limits_{\R^n\,}\!\!\int\limits_{\R^n}\,\frac{|v(y)-v(x)|^{p-2}\bigl(v(y)-v(x)\bigr)\bigl(\phi(y)-\phi(x)\bigr) }{|y-x|^{n+sp}}\,dx\,dy 
  =\,\int_{\Omega}\!f\phi\,dx,
\end{equation}
for all  $\phi \in C^{\infty}_0(\Omega)$. Here $0 < s < 1,\, 1<p<\infty,$ and $f \in L^{\infty}(\R^n).$  All the appearing functions are defined in the whole space $\R^n,$
although the bounded domain $\Omega$ and its boundary $\partial \Omega$ are at the focus. Notice that the boundary values $g$ are prescribed not only on $\partial \Omega$ but in the whole complement $\R^n \setminus \Omega,$ so that the exterior values count.

The equation can be written as
\begin{equation}\label{tre}
  \Lps v(x)\,\equiv\, 2\cdot \lim_{\e \to 0+}\! \underset{|y-x|>\e}{\int}\!\frac{|v(y)-v(x)|^{p-2}\bigl(v(y)-v(x)\bigr) }{|y-x|^{n+sp}}\,dy =\,-f(x)
\end{equation}
for a.\,e. $x\in \Omega,$ provided that the principal value of the integral converges so well that $\Lps v\in L^1_{loc}(\Omega).$ The so defined $\Lps$ is a non-linear and non-local ''fractional differential operator'' often called the fractional $p$-Laplacian. Operators of this type were first studied in \cite{IN10}, but in a slightly different form and in the viscosity sense. The notation
$$\Lps v \,=\, -(-\Delta_p)^sv$$
suggested,  for instance, in \cite{BF14}, \cite{IS14}, \cite{FP14} and \cite{CKP14h} stems from the linear case $p=2$, when one can use the Fourier transform to define
\begin{equation*}
  \widehat{(-\Delta)^s v(x)}  = (2\pi|\xi|)^{2s}\hat{f}(\xi).
\end{equation*}
In this case
\begin{equation*}
  - (-\Delta)^s v(x) = c_{n,s}\,\mathrm{PV}\!\int\limits_{\R^n}\frac{v(y)-v(x)}{|y-x|^{n+2s}}\,dy.
  \end{equation*}
Unfortunately, the simplification offered by the Fourier transform is not available in the non-linear situation $p\not=2$. In principle, the limit as $s \nearrow 1$ leads, under suitable normalization to the much studied $p$-Laplace operator (see for instance Theorem 3.1 in \cite{BPS16}, \cite{BBM01}, and \cite{BBM02}). Then formulas (\ref{ett}), (\ref{tvaa}), and (\ref{tre}) reduce to
\begin{gather}
  J(v)\,=\,\frac{1}{p}\int\limits_{\Omega}|\nabla v|^p\,dx\,-\,\int\limits_{\Omega}fv\,dx,\nonumber\\
\label{glm}  \int\limits_{\Omega}|\nabla v|^{p-2}\nabla v\cdot\nabla \phi\,dx\,=\, \int\limits_{\Omega}f\phi\,dx,\\
  \mathrm{div}\bigl(|\nabla v|^{p-2}\nabla v\bigr)\,=\, -f. \nonumber
\end{gather}

Let us return to the boundary value problem (\ref{eq:mainn}). We assume that
\begin{itemize}
\item $p\in (1,\infty)$\quad \text{and}\quad $s\in (0,1)$
\item $\Omega$ is a bounded domain
\item $ f\in L^{\infty}(\Omega)$
\item $ g \in C(\R^n) \cap L^{\infty}(\R^n)$
\end{itemize}
and nothing else.
We emphasize that $\Omega$ is otherwise quite arbitrary and that unbounded domains could be included upon some minor adjustments.  Its boundary $\partial \Omega$ can even have positive $n$-dimensional volume! The method works well under more general assumptions on the boundary values\footnote{In the classical context with the Laplace equation, Perron's method does not even require that the boundary values be prescribed by a Lebesgue measurable function. In our situation $g$ appears in a  Lebesgue integral taken over the complement $\R^n \setminus \Omega,$ and so $g$ must be measurable. Of course, results are poor under the most general assumptions.}. First, the boundedness of $g$  can be relaxed a lot. Second, the continuity is not needed for Perron's method itself, but for ''arbitrary'' functions $g$ the more interesting properties would fail. 

The main concepts  are the upper class $\u_g,$ which contains supersolutions, and the lower class $\l_g,$ which contains subsolutions. Perron's method produces two pointwise defined functions:
$$\displaystyle \overline{\mathfrak{h}}_g(x) \,=\, \inf\limits_{v\in \u_g}\{v(x)\},\qquad \underline{\mathfrak{h}}_g(x) \,=\, \sup\limits_{u\in \l_g}\{u(x)\}.$$
It follows from the construction that these so-called Perron solutions are ordered:
$$  -\infty \,<\,\underline{\mathfrak{h}}_g(x)\,\leq\,\overline{\mathfrak{h}}_g(x)\,< \,\infty,\quad x \in \Omega.$$ 
They belong to $W^{s,p}_{loc}(\Omega)$ and are continuous except possibly on the boundary $\partial \Omega.$ Our first theorem assures that they are solutions, indeed they are local weak solutions (Theorem \ref{central}). The method is consistent: if the boundary values $g$ also belong to $W^{s,p}(\R^n)$, then the Perron solutions coincide with the unique minimizer of the variational integral (\ref{ett}), but in our work we include more general $g$'s.

The question of \emph{resolutivity} is central: do the Perron solutions coincide? By the constructions, any reasonable solution is squeezed between them. We prove the counterpart to Wiener's celebrated resolutivity theorem in \cite{Wie} (with prescribed continuous boundary values):

\begin{thm}[Wiener]\label{Wiener} $\underline{\mathfrak{h}}_g\,=\,\overline{\mathfrak{ h}}_g.$
\end{thm}

  That the Perron solutions coincide does not mean that they must attain the correct boundary values at all boundary points. (For example, in the range $sp < n$ isolated boundary points become ''ignored''.
  The situation is the same for the ordinary Laplace equation.) We can (for continuous $g$) drop the distinction and write \,$\mathfrak{h}_g = \overline{\mathfrak{h}}_g = \underline{\mathfrak{h}}_g$.\, We say that a given boundary point $\xi_0 \in \partial \Omega$ is \emph{regular}, if
  $$\lim\limits_{x\to\xi_0}\mathfrak{h}_g(x) \,=\,g(\xi_0)\quad \text{for all}\quad g\in C(\R^n)\cap L^{\infty}(\R^n).$$
  We show that \emph{the property of being regular does not depend on the right-hand side of the equation}, see the theorem below and Proposition \ref{propind}. Thus, to discriminate the regular boundary points, it is enough to consider the special case $\Lps v = 0.$ It is  surprising how simple the proof is, while the corresponding result for the $p$-Poisson equation in (\ref{glm}) requires, as it were,  the Wiener criterium, see \cite{Mal96} and \cite{MZ97}.
  \begin{thm}\label{ind} A  boundary point is either regular for all equations
    $$\Lps v(x) = -f(x) \qquad\text{where}\qquad  f\in L^\infty(\Omega)$$
  simultaneously or irregular for all of them. Here $s$ and $p$ are fixed.
  \end{thm}
If $sp > n$ then the Sobolev embedding guarantees the right boundary values, so that all boundary points are regular. We pay attention to the case $sp<n$. \emph{Regularity is equivalent to the existence of a \emph{barrier}}, an auxiliary function, see our Theorem \ref{barrierthm}. An interesting consequence, pointed out by Lebesgue for the Laplace equation, is that the more complement around a boundary point the domain has, the better  the regularity is:

  \begin{prop}\label{Lebesgue} Suppose that we have two domains so that $\Omega_1 \subset \Omega_2$ and a common boundary point $\xi_0 \in \partial \Omega_1 \cap \partial \Omega_2.$ If $\xi_0$ is regular with respect to $\Omega_2$, so it is with respect to $\Omega_1$.
  \end{prop}
  We do not treat the  boundary regularity any further. A finer analysis would require more elaborate tools. Neither do we introduce concepts like  ''$(s,p)$-capacity.'' The merit of our treaty is that also for arbitrary domains we obtain results, even for right-hand sides $f\not \equiv 0$, comparable to corresponding parts in classical Potential Theory. Needless to say, the method itself is not new. We can partly follow the procedure  in \cite{GLM}, which was outlined for equations like ($\ref{glm}$),  and for the Resolutivity Theorem we use some ideas from \cite{LM}. As always, some proofs are straightforward adaptations, but also new, even surprising, difficulties appear. Although, the focus is on the boundary, the exterior values ($\R^n\setminus\overline{\Omega}$) might interfere in a disturbing way, so that a slight change of $g$ far away from $\Omega$ may effect how the "correct" boundary values are attained. On the other hand, sometimes the exterior  helps, as in the surprisingly simple proof of Theorem \ref{ind}. Points in the exterior are always regular, without any requirement on $\dd\Omega$, as expected.
  
  Finally, we mention that the lack of suitable, explicitly known, strict supersolutions has lead us to include a chapter of calculations; in particular the case $sp\leq 1$ is demanding, since the general theory of the Sobolev Spaces $W^{s,p}$ is of little help in this range.
  
  There are several sophisticated problems that are beyond the reach of our present methods.\\

\noindent \emph{Acknowledgements: }T. Kuusi has kindly informed us that the latest version of the manuscript \cite{CKKP15} treats Perron's method even for unbounded boundary values, though restricted to  equations with zero right-hand side. We thank him for this piece of information. The first author is supported by the Swedish Research Council, grant no. 2012-3124. The second author is supported in part by the Research Council of Norway, grant \emph{Waves and Nonlinear Phenomena} (WaNP).

\section{Preliminaries}

In this section we shall present some background.  Preliminary results are also included.
\paragraph{Spaces and notation.}
The fractional Sobolev spaces $W^{s,p}(\R^n)$ with $0<s<1$ and $1<p<\infty$ are assumed to be known by the reader. We refer to "Hitchhiker's guide to the fractional Sobolev spaces" \cite{NPV12} for a good introduction. The norm is defined through
$$\|u\|^p_{W^{s,p}(\R^n)}=\int_{\R^n}\!\!\int_{\R^n} \frac{|u(y)-u(x)|^{p}}{|y-x|^{ sp+n}} d x d y+\int_{\R^n}|u|^p d x.
$$
The space $W^{s,p}_0(D)$ is defined as the closure of $C_0^\infty(D)$ with respect to the norm of $W^{s,p}(\R^n)$. We also note that if $u$ belongs to  $W^{s,p}(\R^n)$,  so do $u_+,\,u_-$, and $|u|.$ The same holds with  $W^{s,p}_0(D)$.

\begin{rem} 
 In general, it is not the same to take the closure in  the norm with the double integral taken only over a subset $D:$
\begin{equation*}\label{Gagliardo}
 [u]^p_{W^{s,p}(D)}=\int_{D\,}\!\!\int_{D} \frac{|u(y)-u(x)|^{p}}{|y-x|^{ sp+n}} d x d y
 \end{equation*}
  However, when $sp\neq 1$ and  $D$ has Lipschitz boundary, the resulting space is the same, see Proposition B.1 in \cite{BLP14}.
 \end{rem}

\begin{lemma}\label{lem:sob0}
Suppose $0\leq \eta\leq \phi$ where $\eta\in W^{s,p}(B)$ and $\phi\in W_0^{s,p}(B)$ where $B$ is a ball. Then $f\in W_0^{s,p}(B)$ as well. 
\end{lemma}
 \begin{proof} From the hypotheses it follows that we can find $\phi_i\in C_0^\infty(B)$ and $0\leq \eta_i\in W^{s,p}(\R^n)$ such that $\phi_i\to \phi$ and $\eta_i\to \eta$ in $W^{s,p}(\R^n)$. Consider the sequence
$$
\psi_i = \min(\phi_i,\eta_i)\in C_0^{0,1}(B).
$$
It is clear that $\psi_i\to \eta$ strongly in $L^p$ and weakly in $W^{s,p}(\R^n)$. Since $W_0^{s,p}(B)$ is closed under weak convergence, it follows that $\eta\in W_0^{s,p}(B)$. 
\end{proof}

\begin{lemma}\label{sob1d} Let $sp>1$ and $f\in W^{s,p}(\R^n)$ be a radial function. Then $f$ is continuous outside the origin.
\end{lemma}
 \begin{proof} It is enough to prove that $f$ is continuous when $1/2<|x|<2$. Denote by $K(x)$ the set consisting of those $y$ so that
$$
|x-y|\leq 2 |x-y_x|=2 ||x|-|y||,
$$
where $y_x$ is the closest point to $x$ such that $|y_x|=|y|$. Then
\begin{align*}
\infty&>\int_{\R^n}\int_{\R^n}\frac{|f(x)-f(y)|^p}{|x-y|^{n+sp}} dy dx\\
& \geq \int_{x\in B_2\setminus B_\frac12}\int_{y\in K(x)\cap B_2\setminus B_\frac12} \frac{|f(x)-f(y)|^p}{|x-y|^{n+sp}}dy dx \\
 &\geq 2^{-n-sp}\int_{x\in B_2\setminus B_\frac12}\int_{y\in K(x)\cap B_2\setminus B_\frac12} \frac{|f(|x|)-f(|y|)|^p}{|x-y_x|^{n+sp}}dy dx\\
 &\geq C\int_{\frac12}^2\int_{\frac12}^2\frac{|f(r)-f(\rho)|^p}{|r-\rho|^{n+sp}}r^{n-1}\rho^{n-1} dr d\rho\\
 &\geq C'\int_{\frac12}^2\int_{\frac12}^2\frac{|f(r)-f(\rho)|^p}{|r-\rho|^{n+sp}}r^{n-1}\rho^{n-1} dr d\rho\\
 &=C'[f(r)]^p_{W^{s,p}([1/2,2])}.
\end{align*}
This means that $f$ seen as a one variable function is in $W^{s,p}([1/2,2])$ and must, therefore, be continuous by Sobolev's embedding if $sp>1$.

 \end{proof}
\paragraph{Notion of solutions.}
To be on the safe side, we define some basic concepts.
\begin{definition}\label{varsol}
We say that a function $u\in\ww$ is a weak  \emph{supersolution} of 
\begin{equation}
\label{eq:main}
\Lps u = -f\text{ in $\Omega$}\\
\end{equation}
if
$$
\int_{\R^n}\!\!\int_{\R^n}\frac{|u(x)\!-\!u(y)|^{p-2}(u(x)\!-\!u(y))(\phi(x)\!-\!\phi(y))}{|x-y|^{n+sp}} dx dy \geq  \int_\Omega f\phi\,dx
$$
for every nonnegative $\phi\in W^{s,p}_0(\Omega)$. A weak \emph{subsolution} is defined by reversing the inequality.

A weak \emph{solution} is a function $u\in \ww$ satisfying
$$
\int_{\R^n}\!\!\int_{\R^n}\frac{|u(x)\!-\!u(y)|^{p-2}(u(x)\!-\!u(y))(\phi(x)\!-\!\phi(y))}{|x-y|^{n+sp}} dx dy = \int_\Omega f\phi\,dx
$$
for every $\phi\in W^{s,p}_0(\Omega)$.
\end{definition}
We shall also need the corresponding local concepts. By a {local weak solution} of the equation one means  a function $u  \in W^{s,p}_{loc}(\Omega) \cap L^{\infty}(\R^n)$ which  satisfies the above equation for all test functions\footnote{One can also allow test functions in $W^{s,p}(K)$ vanishing outside a compact set $K'$, where $K'\subset\subset K\subset\subset \Omega$.} $\phi \in C_0^{\infty}(\Omega)$. The local weak sub- and supersolutions have a similar definition. By regularity theory a (local) weak solution is continuous in $\Omega$, upon a change in a set of measure zero, see Theorem 1.2 in \cite{CKP15b}, Theorem 3.13 in \cite{BP14} (basically Theorem 1.5 in \cite{KMS15}), Theorem 1 in \cite{Lin14} or Theorem 5.4 in \cite{IMS15}, all in slighly different settings. Here, Theorem 5.4 in \cite{IMS15} will do. We shall always assume this continuity. Furthermore, (local) weak supersolutions can be made  lower semicontinuous in $\Omega$ by changing them in a set of measure zero. See \cite{CKKP15} for this regularity result. In the same way, (local) subsolutions are upper semicontinuous\footnote{In \cite{CKKP15}, only a zero right hand side is allowed. See Lemma \ref{nzrhs} how one can arrange for this situation.}.
As a consequence, the functions are defined at every point. In addition, in Corollary 3.6 in \cite{BP14},  the following Caccioppoli inequality for solutions of
$$
\Lps u = -f \quad \text{in } B_R,
$$
is proved:
\begin{align}
\label{cacc}
&\int_{B_r}\int_{B_r} \frac{|u(x)-u(y)|^p}{|x-y|^{n+sp}} dx dy \leq\\
& C\left(\int_{B_R} u^p dx+\left[\left(\int_{\R^n\setminus B_R} \frac{|u(x)|^{p-1}}{|x|^{n+sp}} dx\right)^\frac{1}{p-1}\right]^p+\|f\|_{L^{p*'}(B_R)}^\frac{p}{p-1}\right),\nonumber 
\end{align}
where $C=C(N,p,R,r)$.

We also seize the opportunity to mention the recent papers \cite{CLM12} \cite{KMS15b} where non-local equations of type \eqref{eq:mainn} have been studied. In addition, in the papers \cite{BCF12a}, \cite{BCF12b} and \cite{CJ15} different types of non-local nonlinear equations are introduced.
\paragraph{Existence.} The existence of solutions comes easily from the variational integral $J(v)$ defined in (\ref{ett}).

\begin{thm} \label{thm:varex} Given $g\in \ww$, there is a unique minimizer $u$ of the variational integral $J(v)$ 
in the class $$v\in \ww,\quad v-g\in \wwz.$$ The minimizer is a weak solution in $\Omega$.
\end{thm}
\begin{proof}
  The existence of a minimizer $u$ is standard and follows by the direct method in the calculus of variations. Since the semi-norm is strictly convex, the minimizer is unique. The Euler-Lagrange equation follows from  
$$
\frac{d\,\,}{dt}J(u+t\phi)\Big|_{t=0} = 0, \quad \phi\in W^{s,p}_0(\Omega).
$$
This is the desired equation.
\end{proof}

\begin{rem} In the literature the minimization is sometimes taken over all those $u \in W^{s,p}(\Omega)$ for which  $u=g$  \emph{almost} everywhere in the complement $\R^n\setminus\Omega.$ This will not do here.  Even if $g$ were smooth, this does \emph{not} always yield the same solution in irregular domains. For example, the punctured disk $0 < |x| < 1$ has the boundary point $0,$ which cannot be ignored when $sp >n.$ However, if $\Omega$ has a Lipschitz boundary, these two problems are equivalent, see Proposition B.1 in \cite{BPS16}.
\end{rem}

\paragraph{Comparison Principle.} For two equations
$$\Lps u = -f_1, \qquad \Lps v = -f_2$$
a comparison principle can be formulated, provided that, say, $f_1 \leq f_2 .$ 

\begin{thm}[Variational comparison] \label{thm:varcomp} Let  $u$  and $v$ be two  functions belonging to $W^{s,p}(\R^n).$ Let   $\Omega$ be a bounded domain. If \begin{itemize}
  \item  $v \geq u$ in $\R^n\setminus \Omega$ in the sense that $(u-v)_+\in \wwz$, and
  \item $\Lps v(x) \leq \Lps u(x) \quad\text{when }\, x\in \Omega$ in the sense that
    \begin{align*}
      &\int\limits_{\R^n}\!\!\int\limits_{\R^n}\frac{|v(y)\!-\!v(x)|^{p-2}(v(y)\!-\!v(x) )(\phi(y)-\phi(x))}{|x-y|^{n+sp}}\, dx\, dy \\&\geq  
\int\limits_{\R^n}\!\!\int\limits_{\R^n}\frac{|(u(y)\!-\!u(x)|^{p-2}(u(y)\!-\!u(x)) (\phi(y)-\phi(x))}{|x-y|^{n+sp}}\, dx\, dy 
    \end{align*}
    for all test functions $\phi \geq 0,\, \phi \in W_0^{s,p}(\Omega),$
  \end{itemize}
  then $v \geq u$ also in $\Omega.$ That is, $ v \geq u$ in $\R^n$.
\end{thm}

\begin{proof} 
  This is
 Lemma 9 in \cite{LL14}, where it is shown that 
\begin{equation*}
  u(x) -v(x) = C \quad \text{a.e. \,where} \quad u(x)\geq v(x).
\end{equation*}
If $sp > 1$ we can directly conclude from the Sobolev boundary values that the constant $C$ is zero. If $0 < sp < 1$ one can use the fact that $\psi^+(x)\psi^-(y) = 0,$ where the notation is as in  Lemma 9 in \cite{LL14}.
\end{proof}

\begin{prop}\label{prop:compweak} If  a local weak subsolution  $u$  and  a local weak supersolution  $v$  in $\Omega$ satisfy
  \begin{equation*}
    \begin{cases}
      &u(x)\leq v(x) \quad\text{when}\quad x\in \R^n\setminus \Omega \\ &\limsup\limits_{x\to \xi} u(x)\leq \liminf\limits_{x\to \xi} v(x) \quad \text{when}\quad \xi \in \partial \Omega
      \end{cases}
\end{equation*}
then $u\leq v$ also in $\Omega$.
\end{prop}
\begin{proof} Let $\e>0$. By the hypotheses, the set $\{u>v+\e\}$ is compactly contained in $\Omega.$ (It may be empty.) Therefore, the function $(u-v-\e)^+$ lies in $W_{\text{loc}}^{s,p}(\Omega)$ and vanishes outside $K$ for some $K\subset\subset \Omega$. Plugging in $(u-v-\e)^+$ as a test function  yields the same type of inequalities as in the proof Theorem \ref{thm:varcomp}. We omit the details.
\end{proof}

\begin{definition}[Comparison] We say that the lower semicontinuous function $v$ satisfies the comparison principle from above (for $\Lps v = -f$) in $\Omega$ if whenever 
the subdomain  $D\subset\subset \Omega$ and $h\in C(\overline D)$   is a local weak solution of $\Lps h = -f$ in $D$, such that 
\begin{gather*}
 h(x)\leq v(x) \quad \text{when}\quad x\in \R^n\setminus D\\
 h(\xi)\leq \liminf_{x\to \xi} v(x) \quad \text{when}\quad \xi\in\partial D,
\end{gather*}
then $h\leq v$. A similar definition goes for the comparison principle from below of upper semicontinuous functions $u$.
\end{definition}

Proposition \ref{prop:compweak} states that local weak supersolutions satisfy the comparison principle from above. In fact, the converse is also true.  We are content to prove a special case below, using an idea from \cite{Lin83}. For instance,  the weaker assumption that $v$ is lower semicontinuous and bounded would do. The virtue of the proof is that it does not need the interior continuity of the solution to the obstacle problem, which is known only in the case with zero right-hand side (see \cite{KKP15}).
\begin{prop} Suppose $v\in W^{s,p}(\R^n)\cap C(\R^n)\cap L^\infty(\R^n)$ satisfies the comparison principle from above in $\Omega$. Then $v$ is a local weak supersolution in $\Omega$.
\end{prop}
\begin{proof} Fix $v$ and let $D$ be an open set such that $D\subset\subset \Omega.$ We shall show that $v$ is the solution to an obstacle problem, and as such it is a local weak supersolution.  Let  $u$  be the minimizer of the variational integral $J(u)$ in \eqref{ett} among the class
$$
\mathcal{A}=\{u\in W^{s,p}(\R^n)\cap L^\infty(\R^n),\quad  u\geq v,\quad  u-v\in W_0^{s,p}(D)\},
$$
of admissible functions.  If  only continuous admissible functions are considered, one gets the same $u$. We do not use the knowledge that $u$ is continuous.
We can find a minimizing sequence $u_k\in \mathcal{A}\cap C(\R^n)$, $k=1,2,3,\ldots$ Take $\e > 0.$ The sets $S_{k,\e}=\{u_k>v+\e\}\cap D$ and $S_{k,2\e}=\{u_k>v+2\e\}\cap D$ are open. If they are empty for infinitely many $k$, we have $u\leq v+2\e$. If not, we can find a regular set (satisfying the exterior ball condition) $D_{k,\e}$  such that $$S_{k,2\e}\subset D_{k,\e}\subset S_{k,\e}.$$ Since $D_{k,\e}$ satisfies the exterior ball condition, Lemma \ref{lem:varsolcont} implies that there is a weak solution $h_k\in  W^{s,p}(\R^n)\cap C(\R^n)\cap L^\infty(\R^n)$ satisfying
$$
\begin{cases}
\Lps h_k = -f\quad \text{in }D_{k,\e}\\
h_k = u_k \quad \text{in }\R^n\setminus D_{k,\e}.
\end{cases}
$$
In addition, since $h_k=u_k\leq v+2\e$ outside $D_{k,\e}$, the  comparison principle valid for $v$ implies $h_k\leq v+2\e$ in $\R^n$. We also note that by construction $J(h_k)\leq J(u_k)$. However, the function $h_k$ is not necessarily admissible. The idea is to construct a suitable admissible function, using $h_k$. We note that outside $D_{k,\e}$, $u_k\leq v +2\e$ and as $u_k \in \mathcal{A}$ we have $u_k\geq v$. By Lemma \ref{nzrhs} (proved at the end of the paper), one can modify $h_k$ far away so that it becomes a super- or a subsolution with zero right-hand side. Then by comparing with constant functions we obtain
$$
\|h_k\|_{L^\infty(\R^n)}\leq C(f,\|u_k\|_{L^\infty(\R^n \setminus D_{k,\e})})\leq M<\infty,
$$
where we may assume that $M>\|v\|_{L^\infty(\R^n)}+1$.

Take $\e<1/2$, and fix $\lambda\in (0,\e/(4 M)).$ Consider the convex combination
\begin{equation}\label{convlambda}
w_k=\lambda h_k+(1-\lambda)u_k.
\end{equation}
We note that outside $D_{k,\e}$ we have $w_k = u_k$ and in $D_{k,\e}$
$$
w_k\geq -\lambda M+(1-\lambda)(v+\e)\geq -\frac{\e}4+v-\lambda v+\e-\lambda \e\geq-\frac{\e}4+v-\frac{\e}4+\e-\frac{\e}4 >v,
$$
where we used that $u_k>v+\e$ in $D_{k,\e}$. Hence, $w_k$ is admissible in the minimization so that $J(w_k)\geq J(u)$. By convexity
$$
J(w_k)\leq \lambda J(h_k)+(1-\lambda)J(u_k)\leq J(u_k).
$$
From the inequality
$$J(u)\leq J(w_k)\leq J(u_k)\to J(u)$$
we conclude that the sequence $w_1,w_2,w_3,\dots$ is minimizing. Thus it has a subsequence that converges weakly in $L^p$ to $u$. The subsequence is even strongly convergent in $L^p$, as all minimizing sequences are. We extract a further subsequence that converges a.e. to $u.$
Now the defining identity (\ref{convlambda}) forces the convergence $h_k \to u$ a.e.  Since $h_k\leq v+2\e$ we can conclude that $u\leq v+2\e$. Hence, $u\leq v+2\e$ in any case. Since $\e$ can be chosen how small as we wish, $u\leq v$. By construction also $u\geq v$ and the proof is complete.
\end{proof}

\begin{rem} To proceed to a more general case with a lower semicontinuous $v \in L^{\infty}(\Rn)$ satisfying the comparison principle, one can apply the  Proposition to the infimal concolutions
$$v_{\e}(x)\,=\,\inf_{y}\left\{v(y)+\frac{|x-y|^2}{2 \e}\right\},$$
which obey the comparison principle, if $v$ does. We do not proceed any further here with this matter.
\end{rem}
\section{Exterior Sphere Condition and  Poisson Modification}

When it comes to the continuity of functions in $W^{s,p}(\R^n)$, 
the three cases $sp < n,\,sp=n,\, sp >n$ are  different. When $sp >n$ the space contains merely continuous functions.  For \emph{radial functions}, which we shall use as auxiliary tools, this continuity comes for free also in the wider range  $sp > 1$, see Lemma \ref{sob1d}.

 The case $sp < 1$ is problematic because it contains functions that not at all obey the boundary values. For example, it is straight forward to verify that the characteristic function of the unit ball, 
\begin{equation*}
\mathbf{1}(x) = \begin{cases}
  1\qquad \text{when}\qquad |x| <1,\\
  0\qquad \text{when}\qquad |x| \geq 1,
\end{cases}
\end{equation*}
belongs to $W^{s,p}(\R^n)$ when $sp<1$. In fact, it also belongs to $W_0^{s,p}(B_1)$. 
Furthermore,
$$
\Lps \mathbf{1} \leq -c(n,sp)<0\quad \text{ in $B_1$}. 
$$ 
Note that when $sp<1$, the $W^{s,p}$-seminorm of the characteristic function of a set $D$ is exactly the $sp$-fractional perimeter (cf. \cite{FV13}), which is finite if for instance $\partial D$ is Lipschitz. 

Let $\omega \in W^{s,p}(\R^n)$ denote the unique (variational) solution of the problem
\begin{equation*}
\begin{cases}
\Lps(\omega) = -1,\quad\text{in}\quad B_R\setminus \overline{B_r},\\
\qquad\,\omega = 0,\quad\text{in}\quad B_r,\\
\qquad\,\omega =1,\quad\text{in}\quad \R^n\setminus \overline{B_R}.
\end{cases}
\end{equation*}
where  $B_R\setminus \overline{B_r}$ denotes the ring domain $r < |x| < R.$ In order to use $\omega$ as a barrier \footnote{In the linear case $p=2$ this is Lemma 2.6 in \cite{ROS}, who used superposition and a fractional Kelvin transform. A related result is Lemma 4.3 in \cite{IMS15}}, we need to assure that it attains its boundary values. 

\begin{thm}\label{thm:barrier}
  The function $\omega$ is continuous in $\R^n$ and radial. Furthermore, $\omega > 0$ in $B_R\setminus \overline{B_r}.$
\end{thm}
\begin{proof} 
  By the comparison principle $\omega \geq 0.$ Since $\omega$ is the unique minimizer, it must be a radial function. (Otherwise, a rotation would spoil the uniqueness.) Thus we know that $\omega$ is continuous in $\R^n$ in the case $sp > 1,$ because of the Sobolev inequality in \emph{one} variable, see Lemma \ref{sob1d}. This settles the case  $sp > 1.$

  The boundary values require a  delicate analysis which relies on Lemma \ref{ring}. First, note that $\sup \omega < \infty$ by comparison with the function in (\ref{bounded}). The function
  $$w = (|x|-r)^{\beta}_+,\qquad 0 < \beta < s,$$
  constructed in Lemma \ref{ring} satisfies
  $$\Lps w(x) \leq -1\qquad \text{in a thin ring} \qquad r< |x| < r+\delta.$$
  Multiplying by a large constant $\Lambda$ we see that the function $\Lambda w$ obeys the rules
  \begin{equation*}
    \begin{cases}
      \Lps(\Lambda w) \leq -1 \qquad \text{in} \qquad r < |x| < r+\delta\\
      \quad \Lambda w \geq \|\omega\|_{\infty} \qquad \text{when} \qquad |x| > r + \delta\\
      \quad \Lambda w = 0 \qquad \text{when} \qquad |x| \leq r.
    \end{cases}
  \end{equation*}
  By the comparison principle
  $$0\leq \omega \leq \Lambda w.$$
  Thus $\omega$ is continuous across the sphere $\dd B_r.$

  The outer boundary values are not  essential to us, but we argue as follows. By rotation an arbitrary boundary point can be brought to the position $(-R,0,\dots,0).$
  Now the comparison
  $$1+C_1(R+x_1)_+^{\beta} \geq \omega(x) \geq 1-C_2(R+x_1)_+^{\beta}, \quad 0 < \beta < s, $$
  is possible, see Lemma \ref{half-space}. This yields the continuity across the outer sphere.

  Finally, we have to show that $\omega$ has no zeros in the ring domain. It is convenient to consider the function $\ell(x_1-r)$ from Proposition \ref{minorant} in the half-space $x_1 > r$ Multiply it by a small constant to achieve
  $$\Lps \omega(x) = -1 \leq \Lps \ell(x_1-r)\qquad x_1 > r,\quad r <|x| < R.$$
 By comparison, $\omega(x) \geq \ell(x_1-r).$ This prevents $\omega$ from having zeros in the half-space $x_1> r.$ By rotational symmetry, this is enough.  \end{proof}

\begin{rem} Lemma 4.3 in \cite{IMS15} uses a more implicit construction of the auxiliary barrier needed for Theorem 15.
  \end{rem}

Perron's method presupposes that in some simple domains  the equation can be solved with continuous boundary values, for example in all balls. (For the Laplace equation Poisson's formula originally took care of this initial step.) We shall need this possibility to  construct a so-called Poisson modifcation of our functions.

 \begin{lemma} \label{lem:varsolcont} Assume $g\in W^{s,p}(\R^n)\cap C(\R^n)$. For a ball $B$ the variational solution $u$ of 
\begin{equation*}
\begin{cases}
\Lps u = -f\quad \text{in}\quad B\\
u = g,\quad \text{in}\quad \R^n\setminus B
\end{cases}
\end{equation*}
 is continuous across $\partial B$. Thus $u \in C(\R^n).$
 \end{lemma}

 \begin{rem}\label{ballc2} The ball $B$  can be replaced by any domain $\Omega$ that satisfies the  \emph{exterior sphere condition.} This includes all domains with a boundary of class $C^2.$
 \end{rem}

\bigskip
\begin{proof} Fix a boundary point $\xi_0$. Let $\varepsilon > 0$.  Since $g$ is continuous, there is $\delta>0$ such that
  $$|g(\xi_0)-g(x)| < \varepsilon \qquad \text{when} \qquad |x-\xi_0| < 4\delta.$$
  Let the exterior ball $B(y_0,\delta)$ be tangent to  $B$ at $\xi_0$ so that $ \overline{B(y_0,\delta)}\,\cap\, \overline{B} = \{\xi_0\}.$ Let $\omega$ be the barrier in Theorem  \ref{thm:barrier}, constructed  for the ring $\delta < |x-y_0| < R$, where  $R$ is taken so large that  $B\subset B(y_0,R).$ We can choose a large constant $C$ so that
  $$  |g(x)| < C \omega(x) \qquad \text{when} \qquad |x-y_0| > 2\delta.$$
  It follows that
  $$ g(x)\, < \,\overbrace{g(\xi_0) + \varepsilon + C\omega(x)}^{G(x)} \qquad \text{in}\qquad \R^n\setminus B.$$
  
Now $\Lps(C\omega) = C^{p-1}(-1) \leq -f$\, if\, $C^{p-1}\geq \|f\|_{\infty}.$ By increasing $C$,  if necessary, we can also guarantee that $G(x)$ is a local weak supersolution of \eqref{eq:main} in $B$. We see\footnote{
It is clear that $(g-G)^+\in W_0^{s,p}(B)$. Then, 
$$
0\leq (u-G)^+=(u-g+g-G)^+\leq (u-g)^++(g-G)^+\in W_0^{s,p}(B)
$$
since $u-g\in W_0^{s,p}(B)$. It follows from Lemma \ref{lem:sob0} that $(u-G)^+\in W_0^{s,p}(B)$.} that $(u-G)^+\in W_0^{s,p}(B)$. By the comparison principle (Theorem \ref{thm:varcomp}) $u\leq G$ in $B$, i.e., 
$$
u(x)\leq g(\xi_0)+\e+C\omega (x).
$$
Thus
$$
\limsup_{x\to x_0} u(x)\leq g(\xi_0)+\e+C \lim_{x\to x_0}\omega(x) = g(\xi_0)+\e
$$
and so 
$$
\limsup_{x\to x_0} u(x)\leq g(\xi_0).
$$
Similarly, considering
$$g(x) > g(\xi_0)- \varepsilon - C\omega(x),$$
we get the opposite inequality
$$
\liminf_{x\to x_0}u(x)\geq g(\xi_0).
$$
\end{proof}

Recall that the right-hand member $f$ of our equation is bounded.
Let  $\psi: \Omega \rightarrow \R$  be a bounded, smooth function, say in $C^1(\R^n).$ Select a regular subdomain  $D \subset \subset \Omega.$ By Lemma \ref{lem:varsolcont} and Remark \ref{ballc2} the solution of the problem
\begin{equation*}
  \begin{cases}
    \Lps u = -f \quad \text{in } D\\
    u  = \psi \quad \text{in } \R^n \setminus D
    \end{cases}
\end{equation*}
belongs to $C(\R^n) \cap W^{p,s}(\R^n).$ We use the notation
$$\Psi = P(\psi,D)$$
for this function. Notice that it coincides with $\psi$ outside $D$. For the lack of a better name,
we say that the function $P(\psi,D)$ is the \emph{Poisson modification} of $\psi$ with respect to the domain $D$.

Here $\psi$ was smooth. We want to extend the definition to a
bounded and lower semicontinuous function $v$, obeying the comparison principle from above in $\Omega$. In order to modify it in a ball $B\subset\subset \Omega$ we use an approximation argument.  By semicontinuity there  are smooth functions  $\psi_j$  such that  $$\psi_j \nearrow v\quad\text{in}\quad  \R^n.$$
We can define the {Poisson modification} of $v$ in $B$ as the  limit
$$P(v,B) = \lim_{j \to \infty} P(\psi_j,B).$$ We spell out the details below.

We define 
$$
v_j=P(\psi_j,B)
$$
as the variational solution in $B$ with $v_j-\psi_j\in W_0^{s,p}(B)$. Since $\psi_j$ is smooth,  Lemma \ref{lem:varsolcont} implies that each $v_j$ is continuous across the boundary of the ball $B$. Thus $v_j \in C(\R^n).$ By Theorem \ref{thm:varcomp},   the sequence $v_1,v_2,\dots $ is increasing and $v_j\leq v$. The pointwise limit
$$V(x) = \lim_{j \to \infty}P(\psi_j,B)(x)$$
exists and satisfies $V\leq v$. It is lower semicontinuous. We have
$$
  \liminf_{x\to \xi} V(x)\geq \liminf_{x\to \xi} v_j(x)=\psi_j(\xi),\quad j=1,2,\dots
$$
when $\xi \in \partial B.$ Since $\lim \psi_j(\xi) = v(\xi),$
\begin{equation}\label{Vv}
\liminf_{x\to \xi} V(x)\geq v(\xi)\quad \text{when} \quad \xi \in \partial B.
\end{equation}

In $B$ the constructed function $V$ is the limit of an increasing sequence of solutions. We claim that  $V$ itself is a local weak solution in $B$. 
Since $\psi_j\to v$ and $v$ is bounded, we may assume that the sequence $v_j$ is uniformly bounded. If not, we may consider the uniformly bounded sequence of Lipschitz functions
$$
\max\left(\min\left(v_j,\|v\|_{L^\infty(\R^n)}+1\right),-\|v\|_{L^\infty(\R^n)}-1\right), 
$$
which also converges to $v$. From the Caccioppoli estimate for local weak solutions \eqref{cacc} we have
\begin{align*}
\int_{K}\int_{K} \frac{|u(x)-u(y)|^p}{|x-y|^{n+sp}} dx dy &\leq C(s,p,K)\left(\|v_j\|^p_{L^\infty(\R^n)}+\|v_j\|_{L^\infty(\R^n)}^\frac{p}{p-1}\right)\\&\leq 2C(s,p,K)(\|v\|_{L^\infty(\R^n)}+1)^{\max(p,p/(p-1))}, 
\end{align*}
for any $K\subset\subset B$. This implies that $v_j$ converges weakly in $W^{s,p}(K)$ and strongly in $L^p(K)$. Now pick a smooth test function $\phi\in C_0^\infty(K)$. Since $v_j$ is a local weak solution in $B$, 
$$
\int\limits_{\R^n}\!\!\int\limits_{\R^n}\frac{|v_j(x)\!-\!v_j(y)|^{p-2}(v_j(x)\!-\!v_j(y))(\phi(x)\!-\!\phi(y))}{|x-y|^{n+sp}}\, dx dy =  \int_\Omega f\phi\,dx.
$$
By splitting the integrals into integrals over $K$ and $\R^n\setminus K$ we have
$$
\int_{\R^n}\!\int_{\R^n}=\int_{K}\int_{K}+2\int_{K}\int_{\R^n\setminus K}.
$$
In the integral 
$$
\int\limits_{K}\!\!\int\limits_{K}\!\frac{|v_j(x)-v_j(y)|^{p-2}(v_j(x)-v_j(y))(\phi(x)-\phi(y))}{|x-y|^{n+sp}}\, dx\, dy, 
$$
we may pass to the limit using the weak convergence in $W^{s,p}(K)$. In the integral
\begin{align*}
&2\int_{K}\int_{\R^n\setminus K}\frac{|v_j(x)-v_j(y)|^{p-2}(v_j(x)-v_j(y))(\phi(x)-\phi(y))}{|x-y|^{n+sp}}\, dx dy \\
&=-2\int_{\spt\phi}\int_{\R^n\setminus K}\frac{|v_j(x)-v_j(y)|^{p-2}(v_j(x)-v_j(y))\phi(y)}{|x-y|^{n+sp}}\, dx dy,
\end{align*}
 $|x-y|$ is bounded away from zero, so that we may use the convergence of $v_j$ in $L^p$ to pass to the limit. Hence, 
$$
\int\limits_{\R^n}\!\!\int\limits_{\R^n}\!\frac{|V(x)\!-\!V(y)|^{p-2}\bigl(V(x)\!-\!V(y)\bigr)\bigl(\phi(x)\!-\!\phi(y)\bigr)}{|x-y|^{n+sp}}\, dx\, dy =  \int_\Omega \! f\phi\,dx, 
$$
for any $\phi\in C_0^\infty(K)$. Since $K$ is an arbitrary open set compactly contained in $B$, it follows that $V$ is a local weak solution in $B$·


Next, we prove that $V$ satisfies the comparison principle. Take $D\subset\subset \Omega$ and let $h\in C(\overline D)$ be a local weak solution in $D$ such that 
$h\leq V$ in the complement $\R^n\setminus D$ and 
\begin{equation}
\label{hV}
h(\xi)\leq \liminf_{x\to \xi}  V(x),\quad \xi\in \partial D.
\end{equation}
We claim that $h \leq V.$ Since $V\leq v$ it follows that $h\leq v$ in $\R^n\setminus D$ and 
$$
h(\xi)\leq \liminf_{x\to \xi}  V(x)\leq  \liminf_{x\to \xi}  v(x), \quad \xi\in \partial D.
$$
By the comparison principle, 
\begin{equation}\label{hv} h\leq v.
\end{equation}

In the case $B\cap D$ is empty, the result follows from \eqref{hv}. Therefore, let us study $V$ more closely in the set $B \cap D$ assuming that it is not empty. In $B\cap D$, $V$ is a local weak solution. In $\R^n\setminus D$, $h\leq V$ by the hypothesis on $h$ and in $D\cap (\R^n\setminus  B)$, $h\leq v=V$ by \eqref{hv}. On the boundary we have
\begin{equation*}
  \begin{cases}
h(\xi)\leq \underset{\substack{x\in B\cap D\\x \to \xi}}{\liminf}\, V(x)=V(\xi),\quad \text{when} \quad \xi \in B\cap \partial D\\
 h(\xi) \leq v(\xi) \leq  \underset{\substack{x\in B\cap D\\x \to \xi}}{\liminf}\, V(x),\quad \text{when} \quad \xi \in \partial B \cap D
  \end{cases}
\end{equation*}
where we have used \eqref{hV} and the fact that $V$ is continuous in $B$ for the first inequality and \eqref{hv} together with \eqref{Vv} for the second. By the comparison principle $h \leq V$ in $\R^n$. Thus the comparison principle is valid for $V$. Thus we have proved:

\begin{prop}\label{modification} Let $v$ be a bounded and lower semicontinuous function satisyfing the comparison princple from above in $\Omega$. Then the Perron modification
  $$V =P(v,B)$$
  constructed above, is a solution in $B$ for any ball $B\subset\subset \Omega$, and  obeys the comparison principle in $\Omega$.
  \end{prop}

\section{Perron's Method}\label{Perron}

Consider again the equation
$$\Lps u(x) = -f(x)$$
in the domain $\Omega$, where  $f\in L^{\infty}(\R^n).$ We aim at constructing the solution with given boundary values $g\in  L^{\infty}(\R^n).$ For simplicity we assume also that $g \in C(\R^n).$ Perron's method produces two ordered solutions, which do  or do  not assume the boundary data continuously.

\medskip \noindent
The \emph{upper class} $\u_g$ consists of all functions $v$ such that
\begin{itemize}
\item (i)  $v:\,\R^n \to (-\infty,\infty]$ is lower semicontinuous and obeys the comparison principle from above in $\Omega$. 
\item (ii) $\liminf_{x \to \xi}v(x) \geq g(\xi)\quad \text{when}\quad \xi \in \partial \Omega$,
\item (iii) $v \geq g \quad \text{in}\quad \R^n\setminus \Omega.$
\end{itemize}

\bigskip
\noindent 
The \emph{lower class} $\mathscr{L}_g$ consists of all functions $u$ such that
\begin{itemize}
\item (i)  $u:\,\R^n \to [-\infty,\infty)$ is upper semicontinuous and obeys the comparison principle from below in $\Omega$. 
\item (ii) $\limsup_{x \to \xi}u(x) \leq g(\xi)\quad \text{when}\quad \xi \in \partial \Omega$,
\item (iii) $u \leq g \quad \text{in}\quad \R^n\setminus \Omega.$
\end{itemize}

The upper class contains all lower semicontinuous weak supersolutions satisfying the boundary conditions. If
$v$ belongs to the upper class, so does its Poisson modification $V=P(v,B).$  The property
$$v_1,v_2,\dots,v_k \in \u_g\qquad \Longrightarrow \qquad \min\{v_1,v_2,\dots,v_k\} \in \u_g  $$
for the pointwise minimum is decisive.

 We define pointwise
\begin{align*}
&\text{\sf the upper solution:}\quad  \overline{\h}_g (x) =\inf_{v\in \u_g} v(x),\\ 
&\text{\sf the lower solution:}\quad \underline{\h}_g (x) =\sup_{u\in \mathscr{L}_g} u(x),
\end{align*}
when $x\in \R^n.$

It follows directly that
$$ -\infty \leq \underline{\h}_g (x) \leq \overline{\h}_g (x) \leq +\infty.$$
We shall avoid the cases $\pm \infty.$ As the name suggests, the Perron solutions are solutions, indeed. See Theorem $\ref{central}.$ 
A general observation is that if there exists a local weak solution, say $h_g$ that attains the boundary values at every  point in $\R^n \setminus \Omega$ (in particular on $\partial \Omega$), then 
$$  \underline{\h}_g = h_g = \overline{\h}_g .$$ The reason is that now $h_g$ itself is a member of both classes. Thus the method is consistent.

\begin{rem}\label{bdd} Since $g$ is bounded, \emph{we may restrict ourselves to bounded functions in the upper and lower classes}. In the case with right-hand side $f=0$, we can simply cut the functions with constants.  Otherwise we use some radially decreasing smooth function
\begin{equation*}C(x) =
  \begin{cases}
    1, \qquad |x| < R\\
    0, \qquad |x| > 2R
  \end{cases}
\end{equation*}
(it only matters that $0 < C(x) < 1$ between the spheres). It satisfies
$$\Lps C(x) \leq - \delta \quad\text{when} \quad |x| < R,$$
where
$$ \delta = \min_{|x|\leq R}\,\int\limits_{\R^n}\frac{\bigl(1-C(y)\bigr)^{p-1}}{|y-x|^{n+sp}}\,dy \,\geq
\int\limits_{|y|>2R}\!\frac{\bigl(1-0\bigr)^{p-1}}{\Bigl(\tfrac{|y|}{2}\Bigr)^{n+sp}}\,dy \,>\,0$$  since\, $1-C(y) \geq 0.$\, If we fix $R$ so large that $\Omega \subset B(0,R)$ some
function
\begin{equation}\label{bounded}
  a\,C(x)+b \in \u_g
  \end{equation}
and so every $v$ in $\u_g$ may be replaced by $\min\{v,aC+b\}$ without affecting the upper solution $\overline{\h}_g$. ---The same goes for the lower class and $\underline{\h}_g$. 
\end{rem}

\paragraph{Continuity of the Perron  Solutions.}
We shall prove that the Perron solutions satisfy the differential equation. Our first step is to establish  continuity at interior points.

\begin{prop}{\label{prop:cont}} The upper and the lower Perron solutions are continuous functions in $\Omega$.
\end{prop}
\begin{proof} By symmetry, it is enough to write the proof for the upper solution. Let $x_0\in \Omega$ and take $\e>0$. We will then show that if $r$ is small enough, 
$$
\osc_{B_r(x_0)} \{\overline{\h}_g\}\,\leq \,\e.
$$
Take $x_1,x_2\in B_r(x_0)$. We can then find two  decreasing sequences  $v_i^1$ and $v_i^2$ of functions  in $\u_g$ so that 
$$
v_i^1(x_1)\to \overline{\h}_g(x_1),\quad v_i^2(x_2)\to \overline{\h}_g(x_2).
$$
Then the functions $v_i = \min\{v_i^1,v_i^2\}$ belong to $\u_g$ and satisfy
$$
v_i(x_1)\to \overline{\h}_g(x_1),\quad v_i(x_2)\to \overline{\h}_g(x_2).
$$ 

Take a larger ball $B_R(x_0)$ and use the Perron modification $V_i = P(v_i,B_R)\in \u_g$. By construction, $\overline{\h}_g\leq V_i\leq v_i$.
We can now choose the index $i$ so that 
$$
v_i(x_1)\leq \overline{\h}_g(x_1)+\frac{\e}{2},\quad v_i(x_2)\leq \overline{\h}_g(x_2)+\frac{\e}{2}.
$$
Then
$$
\overline{\h}_g(x_1)-\overline{\h}_g(x_2)\leq V_i(x_2)-V_i(x_1)+\frac{\e}{2}.
$$
Now, from Theorem 5.4 in \cite{IMS15}, we have for $x_1,x_2 \in B_r(x_0),$
$$
|V_i(x_1)-V_i(x_2)|\leq C_R\left(\|f\|^\frac{1}{p-1}_{L^\infty(\Omega)}+\|V_i\|_{L^\infty(\R^n)}\right)r^\alpha\leq Cr^\alpha
$$
since we can assume that the sequence $v_i$ is uniformly bounded (cf. Remark \ref{bdd}), so that $V_i$ is uniformly bounded as well. By choosing $r$ small enough, we thus have
$$
\overline{\h}_g(x_1)-\overline{\h}_g(x_2)\leq \frac{\e}{2} + \frac{\e}{2}  =  \e.
$$
Letting  $x_1$ and $x_2$ change places, we  complete the proof.

\end{proof}

\paragraph{A central property.}  We are ready to prove:
\begin{thm}[Perron] \label{central} The upper and the lower Perron solutions are local weak solutions in $\Omega$.
\end{thm}
\begin{proof} We prove it only for $\overline{\h}_g$. First we construct a monotonically decreasing sequence of functions $w_i\in \u_g$  converging to $\overline{\h}_g$ at each rational point in  $\R^n$. Let $q_1,q_2,\ldots$ be an enumeration of the rational points. At each rational point $q_k$ there is a sequence $v^k_1,v^k_2,v^k_3,\dots$   in $\u_g$ so that $$v^k_i(q_k)\to \overline {\h}_g(q_k)\quad \text{as}\quad i\to \infty.$$ The function
$$
w_i = \min \{v^1_i,\ldots, v^i_i\}
$$
also belongs to $\u_g$, and in particular
$$
\overline{h}_g\leq w_i\leq v^k_i, \qquad k=1,\ldots,i.
$$
Hence, $w_i\to \overline{\h}_g$ at all rational points.

Let
$$
W_i = P(w_i,B),
$$
where $B$ is a ball compactly contained in $\Omega$. We claim that also the sequence $W_1,W_2,W_3,\dots   $ converges to $\overline {\h}_g$ at the rational points.  By Proposition \ref{modification}, $W_i\leq w_i$ and $W_i\in \u_g$. Hence,  
$$\overline{\h}_g \leq W_i \leq w_i$$
and the claim follows.
Arguing as in the proof of Proposition \ref{modification} one can prove that the pointwise limit
$$
W=\lim_{i\to \infty} W_i
$$
is a local weak solution in $B$. Thus $W$ is continuous in $B$ by Theorem 5.4 in \cite{IMS15}.  
By Proposition \ref{prop:cont}  $\overline{\h}_g$ is continuous in $B.$ We have the situation that  $W=\overline{\h}_g$ at all rational points in $B$, hence by the continuity
at all  points in $B.$  Thus $\overline{\h}_g$ is a local weak solution in $B.$  Since $B$ was arbitrary,  $\overline{\h}_g$ is a local weak solution in $\Omega$. 

\end{proof}

\section{Boundary Values and Barriers}

Consider some boundary point $\xi_0 \in \partial \Omega$. The boundary point is called \emph{regular} if
\begin{equation*}
  \lim_{\underset{x\to\xi_0}{ x \in \Omega}} \overline{\h}_g(x) = g(\xi_0) = \lim_{\underset{x\to\xi_0}{ x \in \Omega}} \underline{\h}_g(x)
  \end{equation*}
for \emph{all} continuous $g$. (By Wiener's resolutivity theorem below, $\overline{\h}_g =\underline{\h}_g$.)
 From the proof of Lemma $\ref{lem:varsolcont}$
we can read off the following sufficient condition for regularity, the so-called \emph{exterior sphere condition}:
\begin{prop}[Exterior Sphere Condition] Suppose that there is a ball $B_\rho(y_0)$ such that $B_\rho(y_0)\cap \overline\Omega =\{\xi_0\}$. Then
$$
\lim_{x\to \xi_0}\overline{\h}_g(x)=g(\xi_0).
$$
The same holds for the lower solution.
\end{prop}

It is plain that if the exterior sphere condition holds at each boundary point, then the Perron solutions coincide: $\overline{\h}_g = \underline{\h}_g$ for continuous $g$.

Continuous "boundary values" can never fail at points away from $\partial \Omega.$ 

\begin{prop} [Exterior Values] If $x_0\in \R^n\setminus \overline \Omega$  
$$
\lim_{x\to x_0}\overline{\h}_g(x)=g(x_0).
$$
The same holds for the lower solution. 
\end{prop}
\begin{proof} Let $x_0\in \R^n\setminus \overline\Omega$ and fix $\e>0$. Since $g$ is continuous, there is $\delta>0$ such that
 $$|g(x)-g(x_0)|<\e$$ 
in $B_\delta(x_0)$ and $B_\delta(x_0)\cap \overline \Omega =\emptyset$. Choose $R>0$ such that $\Omega\subset B_R(x_0)$ and let $\omega$ be the function in Theorem \ref{thm:barrier} with $B_r$ replaced by $B_{\delta/2}(x_0)$ and $B_R$ replaced by $B_R(x_0)$. Then $\omega$ is a strictly positive function in $\R^n\setminus B_\delta(x_0)$. Consequently, we can find a $C>0$ such that  
$$
|g(x)|\leq C\omega(x), \text{ for $x\in \R^n\setminus ( B_\delta(x_0)\cup \Omega)$}.
$$
Therefore, 
$$
g(x)\leq g(x_0)+\e +C\omega(x), \text{ in $\R^n\setminus \Omega$}.
$$
By chooing $C$ larger if needed, we can also guarantee that $C\omega$ is a local weak supersolution of \eqref{eq:main} in $\Omega$. Therefore, the function $g(x_0)+\e +C\omega(x)\in \u_g$ so that
$$
\overline{\h}_g \leq g(x_0)+\e +C\omega , \text{ in $\R^n$}.
$$
Similarly, we can prove
$$
g(x_0)-\e +C\omega\leq \underline{\h}_g , \text{ in $\R^n$}.
$$
In particular, at the point $x=x_0$, we have since $\omega(x_0)=0$
$$
g(x_0)-\e\leq \underline{\h}_g(x_0)\leq \overline{\h}_g(x_0) \leq g(x_0)+\e.
$$
Since $\e$ is arbitrary, this yields the desired result.
\end{proof}
This last result in the exterior is valid without any assumptions on the boundary of $\Omega$.

\begin{definition} We say that the function $\gamma$ is a \emph{barrier} at the point $\xi_0 \in \partial \Omega,$\, if
  \begin{itemize}
  \item{(i)} $\gamma$ is continuous in $\R^n$
  \item{(ii)} $\gamma$ is a weak supersolution in $\Omega$
  \item{(iii)} $\gamma(x)\,>\,0$ when $x\not=\xi_0$ and
    $$\lim_{x \to \xi_0}\gamma(x)\,=\,0.$$
  \end{itemize}
\end{definition}

In passing, we see that the expedient function in Theorem \ref{thm:barrier} does not quite  satisfy the definition, since it is zero in a small ball. (Actually, we used infinitely many functions there.) 

\begin{thm}[Lebesgue]\label{barrierthm} The boundary point $\xi_0$ is regular if and only if there exists a barrier at  $\xi_0$.
\end{thm}

\begin{proof} That the existence of a barrier is sufficient for regularity follows by the same reasoning as in the proof of Lemma \ref{lem:varsolcont}. To see this, just replace  $\omega$ there by $\gamma.$

  For the necessity, we construct a barrier by solving the problem
  \begin{equation*}
    \begin{cases}
      \Lps \gamma = -f,\\
      \gamma = g \quad\text{in}\quad \R^n\setminus \Omega,
    \end{cases}
  \end{equation*}
where we select the the boundary values $g$ so that they satisfy 
$$
\Lps g \geq \|f\|_{L^\infty(\Omega)} \quad \text{ in $\Omega$}.
$$
To this end we may assume that $\Omega\subset B_1$ and that $\xi_0 = 0$. One may use
$$
g(x)=\Lambda (1-(1-|x|)_+^\frac{s}{2}).
$$
for $\Lambda$ large enough. Indeed, by Corollary \ref{half-infimum}
$$
\Lps\, (1-|x|)_+^\frac{s}{2}\leq -C\quad \text{ in $B_1$}.
$$
Then
  $$\lim_{x\to \xi_0}\gamma(x) = g(\xi_0)$$
  by the regularity assumption on $\xi_0$. By the comparison principle $\gamma(x) \geq g(x) > 0$ when $x \not= \xi_0.$ This shows that $\gamma$ will do as a barrier.
\end{proof}

Now Proposition \ref{Lebesgue} in the Introduction follows immediately.

 \section{Wiener's Resolutivity Theorem}

 The resolutivity theorem, originally formulated in 1925 by N. Wiener for harmonic functions (cf. \cite{Wie}), states that the Perron solutions coincide under fairly general conditions: \emph{the domain is arbitrary}, the boundary values are continuous, and the right-hand side is bounded.

 Let $\Omega$ be a bounded domain. (We allow its boundary $\partial \Omega$ to be arbitrary; it may even have positive $n$-dimensional volume.) Consider the equation
 $$\Lps u\, = -f\qquad\text{in} \quad \Omega$$
 with $f\in L^{\infty}(\Omega).$ The boundary values $g$ are continuous and bounded: $g\in C^{\infty}(\R^n)\cap L^{\infty}(\R^n).$ Perron's method produces two solutions such that
 $$-\infty < \, \underline{\h}_g \,\leq\, \overline{\h}_g \,< \infty.$$ We shall prove that\, $\underline{\h}_g = \overline{\h}_g $. This is called resolutivity and the common solution is denoted by $\h_g$.

 \paragraph{The obstacle problem.} We use an obstacle problem as an auxiliary tool.  Given a smooth $\psi\in  L^{\infty}(\R^n)$ we let $\psi$ act as an \emph{obstacle} in order to obtain a suitable supersolution. Let $D$ be a domain and consider the problem of minimizing the variational integral
$$
J(u)=[u]^p_{\ww}-\int\limits_D \! f u\,dx
$$
among all functions in the class $$\{u\in W^{s,p}(\R^n),\quad u\geq \psi\,\,\text{in}\,\, D, \quad u-\psi\in W_0^{s,p}(D)\}.$$
Here the obstacle $\psi$ also induces the boundary values.

By the direct method in the Calculus of Variations, the existence of a unique minimizer $v$ is established. 
It is plain to verify that the solution of the obstacle problem is a weak supersolution. Thus it is lower semicontinuous by \cite{CKKP15}, again by the modification in Lemma \ref{nzrhs}. It is a solution in the open set $\{v>\psi\}\cap D$ where the obstacle does not hinder (under our assumption that $\psi$ is continuous). This obstacle problem has been studied more closely in \cite{KKP15}.
\paragraph{The resolutivity.} We turn to the resolutivity question. For the proof of Theorem \ref{Wiener}
we assume that  $g\in C(\R^n)\cap L^\infty(\R^n)$ and $f\in L^{\infty}(\Omega)$. We claim that the upper and the lower solutions of \eqref{eq:mainn} coincide.

\begin{proof}[~Proof Theorem \ref{Wiener}]
  A reduction to the case when $g\in C^{\infty}(\R^n)$ is possible. To see this, let $\e > 0.$ There is a function $\phi \in C^{\infty}(\R^n)$ such that
  $$g(x)-\e \leq \phi(x) \leq g(x) + \e$$
  at every point $x.$ If we know that $\overline{\h}_{\phi} = \underline{\h}_{\phi}$ , then
  $$\overline{\h}_g \leq \overline{\h}_{\phi+\e} = \overline{\h}_{\phi} + \e = \underline{\h}_{\phi} +\e \leq \underline{\h}_{g+\e}+\e = \underline{\h}_g+ 2\e.$$
  We deduce that $\overline{\h}_g \leq \underline{\h}_g.$ Always,  $\overline{\h}_g \geq \underline{\h}_g$ and so $\overline{\h}_g = \underline{\h}_g.$ Therefore we may assume that  $g\in C^{\infty}(\R^n)$.

  The proof is based on the expedient fact that the solution to the equation with the boundary values $g$ taken merely in the sense of the Sobolev norm $W^{s,p}_0(\Omega)$ is \emph{unique}. Needless to say, this solution does not have to belong to any of the classes $\u_g$ and $\l_g$: it may ignore pointwise described boundary values. An obstacle problem makes it possible to control the boundary values in the procedure. 
  
  Let the function $g$ act as an obstacle and consider the problem of minimizing the integral ($\ref{ett}$) among all functions $u\geq g$ in $\R^n$ belonging to $W^{s,p}(\R^n)$ with boundary values $u-g \in W^{s,p}_0(\Omega)$. Let $v$ denote the unique minimizer. It is decisive that   $v-g \in W^{s,p}_0(\Omega)$ and that $v\geq g$.  In addition, $v$ is a local weak supersolution in $\Omega$ and therefore $v\in \u_g$. 

  Take  an exhaustion of $\Omega$ with regular domains $D_j$:
  $$D_1 \subset D_2 \subset D_3 \subset \cdots\subset  \Omega \,= \,\bigcup\limits_{j=1}^{\infty}D_j.$$
  To us it is enough that each $D_j$ satisfies the exterior sphere condition so that  the Poisson modifications   $V_j=P(v,D_j)$
  are continuous in $\Omega$. By construction, $V_j-v\in W_0^{s,p}(\Omega)$ and $V_j\in \u_g$. Now
  \begin{align*}
    & v\geq V_1\geq V_2\geq V_3\geq\cdots\\  &\|g\|_{W^{s,p}(\R^n)}\, \geq \, \|V_1\|_{W^{s,p}(\R^n)}\, \geq \,  \|V_2\|_{W^{s,p}(\R^n)}\, \geq \cdots
    \end{align*}
   Hence the pointwise limit $W = \lim V_j$  exists and, in addition, the $V_j$s converges locally weakly in $W^{s,p}(\R^n)$ and locally in $L^p$.  Repeating the arguments in the proof of Proposition \ref{prop:cont}, one can prove that, being locally the limit  of solutions,   $W$ itself is a local weak solution in $\Omega$. 

   Since the space  $W_0^{s,p}(\Omega)$ is closed under weak convergence 
   \\ $W-g\in W_0^{s,p}(\Omega)$. Hence, $W\in W^{s,p}(\R^n)$ is a local weak solution in $\Omega$ and satisfies $u-g\in W_0^{s,p}(\Omega)$. There is only one such function $W$ (Theorem \ref{thm:varex}). Notice carefully that $W$ may fail to belong to the upper class $\u_g$. But what counts now is that  each $V_j \in \u_g$, and so, at every point
\begin{equation*}
  \overline{\h}_g(x) \leq V_j(x),\qquad \overline{\h}_g(x) \leq \lim_{j \to \infty}V_j(x) = W(x).
  \end{equation*}
   Repeating the procedure using subsolutions from below we arrive at
   $$\underline{\h}_g(x) \geq W(x)$$
   with\emph{ the same unique solution} $W$ to the Dirichlet  problem with boundary values in Sobolev's sense. We conclude that $\overline{\h}_g = \underline{\h}_g.$
   \end{proof}
\begin{rem} The proof above reveals that the exterior limit
$$
\underset{\substack{x\in \R^n\setminus\Omega \\x \to \xi}}{\lim}\,\h(x)=g(\xi), 
$$   
holds for every $\xi \in \partial \Omega$. Thus the correct boundary value $g(\xi)$ may fail only for the interior limit.
\end{rem}
   
\section{Changing the Right-Hand Side}
In this section we show how one can transform a supersolution of $\Lps u\leq 1$ to a supersolution of $\Lps u\leq 0$ by pulling it down far away. 

\begin{lemma}\label{simple} Suppose $a\geq -2$ and $M\geq \max(3,a)$. Then
$$
|a+M|^{p-2}(a+M)-|a|^{p-2}a\geq c_pM^{p-1}.
$$ 
\end{lemma}
\begin{proof} $\qquad
|a+M|^{p-2}(a+M)-|a|^{p-2}a =(p-1)\int_0^M |a+s|^{p-2} ds. 
$  \end{proof}

\begin{lemma}\label{nzrhs} Suppose $u$ is a local weak supersolution of $\Lps u\leq 1$ in $B_1.$ If $u\geq -1$ in $B_1$ and $|u|\leq 1$ in $\R^n\setminus B_1$ then the function $$\tilde u = u(x)-M\eta(x)$$ is a supersolution of $\Lps u\leq 0$ in $B_1$, provided that the constant $M$ is large enough. Here $\eta$ is a smooth function satisfying
$$
\eta=
\begin{cases}
0\quad \text{in}\quad & B_2\\
1\quad \text{in} \quad & \R^n\setminus B_4
\end{cases}
$$
and $0\leq \eta \leq 1.$
 \end{lemma}
 \begin{proof}   Let $\phi$ be a non-negative test function in $C_0^\infty(B_1)$.  We claim that
 \begin{align*}
   & \int_{\R^n}\!\!\int_{\R^n}\frac{|\tilde u(x)-\tilde u(y)|^{p-2} (\tilde u(x)-\tilde u(y))}{|x-y|^{n+sp}}(\phi(x)-\phi(y))\, dx dy\\
   &\geq
   \int_{\R^n}\!\!\int_{\R^n}\frac{|u(x)-u(y)|^{p-2}(u(x)-u(y))}{|x-y|^{n+sp}}(\phi(x)-\phi(y))\, dx dy\\
     & + C(s,p,n)M^{p-1}\int_{\R^n}\phi(x)\,dx,
 \end{align*}
 from which the desired result can be read off. To this end,
 we now compare the corresponding integrals.  First of all, we can split the integrals as
  $$
  \int_{\R^n}\int_{\R^n}= \int_{B_2}\int_{B_2}+2\int_{\R^n\setminus B_2}\int_{B_2}=\int_{B_2}\int_{B_2}+2\int_{\R^n\setminus B_2}\int_{B_1}
 $$
 since $\phi$ has compact support in $B_1$. The integrals over $B_2 \times B_2  $ remain the same.  Notice that $(\phi(x)-\phi(y))\,dxdy = \phi(x)\,dxdy$ when $y \not \in B_2$. We note that 
 \begin{align*} 
& \int_{\R^n\setminus B_2}dy\!\int_{B_1}\frac{|\tilde u(x)-\tilde u(y)|^{p-2}(\tilde u(x)-\tilde u(y))}{|x-y|^{n+sp}}\phi(x)\, dx \geq \\
 &\int_{B_4\setminus B_2}dy\!\int_{B_1}\frac{| u(x)- u(y)|^{p-2}( u(x)- u(y))}{|x-y|^{n+sp}}\phi(x)\, dx \\
 &+\int_{\R^n\setminus B_4}dy\!\int_{B_1}\frac{|\tilde u(x)-\tilde u(y)|^{p-2}(\tilde u(x)-\tilde u(y))}{|x-y|^{n+sp}}\phi(x)\, dx. 
 \end{align*}
 The inequality above was valid pointwise, since $\tilde u(x) - \tilde u(y) \geq u(x)-u(y)$ when $x \in B_1,\, y \in B_4 \setminus B_2.$
 The last integral can be estimated as
\begin{align*}
 &\int_{\R^n\setminus B_4}dy\!\int_{B_1}\frac{|\tilde u(x)-\tilde u(y)|^{p-2}(\tilde u(x)-\tilde u(y))}{|x-y|^{n+sp}}\phi(x)\, dx \\
 &= \int_{\R^n\setminus B_4}dy\!\int_{B_1}\frac{|u(x)-u(y)+M|^{p-2}( u(x)- u(y)+M)}{|x-y|^{n+sp}}\phi(x)\, dx \\
 &\geq \int_{\R^n\setminus B_4}dy\!\int_{B_1}\frac{|u(x)-u(y)|^{p-2}( u(x)- u(y))}{|x-y|^{n+sp}}\phi(x)\, dx \\&+\int_{\R^n\setminus B_4}dy\!\int_{B_1}\frac{c_pM^{p-1}}{|x-y|^{n+sp}}\phi(x) \,dx \\
 &\geq \int_{\R^n\setminus B_4}dy\!\int_{B_1}\frac{| u(x)- u(y)|^{p-2}( u(x)- u(y))}{|x-y|^{n+sp}}\phi(x)  dx\\&+C(s,p,n)M^{p-1}\int_{B_1}\phi(x)\, dx, 
\end{align*}
where we used Lemma \ref{simple} and the fact that $u(x)-u(y)\geq -2,$ taking $M\geq \max(3,2\|u\|_{L^\infty(\R^n)})$. The claim follows.
\end{proof}

\paragraph{Regular Points}
As a corollary of the resolutivity theorem, we obtain Theorem \ref{ind}, according to which the right-hand side $f$ of the equation $\Lps u = f$ has no influence of the regularity of a boundary point. This is an immediate consequence of the proposition below. We keep $s,p,$ and $\Omega$ fixed, but change the right-hand side.
\begin{prop}\label{propind} Suppose $|f|\leq 1 $. A boundary point is regular for the equation with the right-hand side   $-f$ if and only if it is regular with the right-hand side is $-1$.
\end{prop}
\begin{proof} Suppose first that $\xi_0$ is regular with respect to $-1$. Denote by $u^\pm$ the (unique Perron) solutions of 
$$
\begin{cases}
\Lps u^\pm =-\pm 1 \quad\text{in } \Omega,\\
u^\pm = g \quad\text{in } \R^n\setminus \Omega,
\end{cases}
$$
and by $u$ the solution of
$$
\begin{cases}
\Lps u =-f \quad\text{in } \Omega,\\
u = g \quad \text{in } \R^n\setminus \Omega.
\end{cases}
$$
Then $u^-\leq u\leq u^+.$  Since $\xi_0$ is regular with respect to $-1$ it is also regular with respect to $+1$. Hence, 
$$
g(\xi_0) =  \lim_{x\to \xi_0} u^-(x)\leq \lim_{x\to \xi_0} u(x)\leq \lim_{x\to \xi_0} u^+(x)=g(\xi_0), 
$$
so that $\xi_0$ is regular with respect to $-f$.

Now suppose $\xi_0$ is regular with respect to $-f$.  Let $u$ be the solution of
$$
\begin{cases}
\Lps u =-f  \quad\text{in } \Omega,\\
u = g \quad \text{in } \R^n\setminus \Omega.
\end{cases}
$$
Denote by $v$ the solution of 
$$
\begin{cases}
\Lps v =-1  \quad\text{in } \Omega,\\
v = g \quad \text{in } \R^n\setminus\Omega.
\end{cases}
$$
Then $v\geq u$ so that 
$$
\liminf_{x\to \xi_0} v(x)\geq \liminf_{x\to \xi_0} u(x)=g(\xi_0),
$$
since $\xi_0$ is regular with respect to $-f$. Now consider $\tilde v = v+M\eta$, where $\eta\geq 0$ is a smooth function supported at a positive distance $\Omega$. Then, as in Lemma \ref{nzrhs}, for $M$ large enough $\Lps \tilde v \geq 1$ and $\tilde v = \tilde g=g+M\eta$ in $\R^n\setminus \Omega$. Let $w$ be the solution of 
$$
\begin{cases}
\Lps w =-f  \quad\text{in } \Omega,\\
w = \tilde g \quad \text{in } \R^n\setminus \Omega.
\end{cases}
$$
Then $w\geq \tilde v$ so that
$$
\limsup_{x\to \xi_0}\tilde  v(x)\leq \limsup_{x\to \xi_0} w(x) = g(\xi_0), 
$$
again since $\xi_0$ is regular with respect to $-f$. Since $\tilde v = v$ near $\dd\Omega$ we conclude
$$
\lim_{x\to \xi_0} v(x)= g(\xi_0).
$$
Therefore, $\xi_0$ is regular with respect to $-1$.
\end{proof}

\section{Explicit ''Barriers''}

In the non-linear case there are few explicit examples of $sp$-harmonic functions in the literature. In Lemma 3.7 in \cite{BPS16} we find that
$$\Lps(\langle a,x\rangle) = 0,\qquad s > 1-\frac{1}{p}$$
in $\R^n$. (The computations diverge without the restriction on $s$.) 
See also \cite{Dyd12} for explicit computations when $p=2$.
\paragraph{Concave Functions.} 
Suppose $v$ is a concave function. From
$$
v(x)-v(y)\geq \nabla v(x)\cdot (x-y).
$$
we obtain 
\begin{align*}
-\Lps v =2\int_{\R^n} \frac{|v(x)-v(y)|^{p-2}(v(x)-v(y))}{|x-y|^{n+sp}}\, dy \\\geq 2\int_{\R^n} \frac{|\nabla v(x)\cdot(x-y)|^{p-2}(\nabla v(x)\cdot(x-y))}{|x-y|^{n+sp}}\, dy = 0, 
\end{align*}
since the integrand is an odd function with respect to the variable $x-y$, given that the integral at infinity converges,\footnote{A careful arrangement of the calculation shows that
  $$\underset{\e\leq |x-y|\leq L}{\int\!\!\int} \frac{|v(x)-v(y)|^{p-2}(v(x)-v(y))(\phi(y)-\phi(x))}{|x-y|^{n+sp}}\,dx\,dy \,\geq\,0$$
  for every non-negative $\phi \in C^{\infty}_0(\R^n).$} of course. Hence \emph{the concave function is a supersolution}.

\paragraph{Infimal Convolution.}
In passing, we mention that the so-called \emph{infimal convolution}
$$v_{\e}(x) = \inf\Bigl\{v(y) + \frac{|x-y|^2}{2\e}\Bigr\}$$
preserves the inequality $\Lps v \leq -1$. This holds under fairly general assumptions.

\subsection{The Positive Part $(x_1^+)^s$}

The function
$$u(x_1,x_2,\dots,x_n)   = (\mathbf{n}\cdot x)_+^s$$
is a solution to the equation $\Lps u = 0$ in the half-space 
$$\mathbf{n}\cdot x = n_1x_1+\dots+n_nx_n > 0.$$
See Lemma 3.1 in \cite{IMS15}. We shall also need a strict supersolution of this type. We begin with the one-dimensional case.

\begin{lemma} In one dimension
  $$\Lps (x_+)^{\beta} = -C(\beta,s,p)\,x^{\beta(p-1)-sp}\quad\text{when}\quad x >0,$$
  where the constant $C(\beta,s,p) > 0$ for $0 < \beta < s.$ In the case $\beta = s$ we have
  $$\Lps (x_+)^s = 0 \quad\text{when}\quad x >0.$$
\end{lemma}

\begin{proof}

Let $u(y) = (y_+)^{\beta}$. Keep $0<x<\infty.$
We split the integral in the formula 
$$\Lps u(x) = 2\int_{-\infty}^{\infty}\!\frac{|u(y)-u(x)|^{p-2}\bigl(u(y)-u(x)\bigr)}{|y-x|^{1+sp}}\, dy$$
into three parts:

\medskip
\noindent {\bf The part $y\leq 0$: }
$$\int_{-\infty}^{0}\!\frac{|u(y)-u(x)|^{p-2}\bigl(u(y)-u(x)\bigr)}{|y-x|^{1+sp}}\, dy = -\frac{x^{\beta(p-1)-sp}}{sp}$$

\bigskip
\noindent {\bf The part $0\leq y\leq x-0$:}\\ \noindent   Write $x_{\e}= x-\e.$ We have
  \begin{align*}
  &\int_{0}^{x_{\e}}\!\frac{|y^{\beta}-x^{\beta}|^{p-2}\bigl(y^{\beta}-x^{\beta}\bigr)}{|y-x|^{1+sp}}\,dy
\,  =\,  - \int_{0}^{x_{\e}}\!\frac{(x^{\beta}-y^{\beta})^{p-1}}{|x-y|^{1+sp}}\,dy \\&=\,-
    \Bigl[(x^{\beta}-y^{\beta})^{p-1}\frac{(x-y)^{-sp}}{+sp}\Bigr]_{0}^{x_{\e}}\,\, -\,\tfrac{\beta(p-1)}{sp}\int_{0}^{x_{\e}}\!(x-y)^{-sp}y^{\beta -1}(x^{\beta}-y^{\beta})^{p-2}\,dy\\&=
    +\frac{{x_{\e}}^{\beta(p-1)-sp}}{sp} - (x_{\e}^{\beta}-y^{\beta})^{p-1}\frac{(x_{\e}-y)^{-sp}}{sp}\\& - \tfrac{\beta(p-1)}{sp}x^{\beta(p-1)-sp}\int_{0}^{1-\frac{\e}{x}}\!(1-t)^{-sp}t^{\beta -1}(1-t^{\beta})^{p-2}\,dt,
  \end{align*}
  where we have substituted $y =tx,\,dy = xdt.$
  \medskip

\noindent {\bf The part $x+0\leq y$:}\\ \noindent   Write $x^{\e} = x+\e.$
  

 \begin{align*}
&\int_{x^{\e}}^{\infty}\!\frac{|y^{\beta}-x^{\beta}|^{p-2}\bigl(y^{\beta}-x^{\beta}\bigr)}{|y-x|^{1+sp}}\,dy
   =  \int_{x^{\e}}^{\infty}\!\frac{(y^{\beta}-x^{\beta})^{p-1}}{|y-x|^{1+sp}}\,dy \\&=
   \Bigl[(y^{\beta}-x^{\beta})^{p-1}\frac{(y-x)^{-sp}}{-sp}\Bigr]_{x^{\e}}^{\infty}\,\,+\, \tfrac{\beta(p-1)}{sp} \int_{x^{\e}}^{\infty}\!(y-x)^{-sp}y^{\beta -1}(y^{\beta}-x^{\beta})^{p-2}\,dy\\&
   = (y^{\beta}-x^{\e\,\beta})^{p-1}\frac{(y-x^{\e})^{-sp}}{sp} + \tfrac{\beta(p-1)}{sp}x^{\beta(p-1)-sp}\int_{1+\frac{\e}{x}}^{\infty}\!(\tau-1)^{-sp}\tau^{\beta-1}(\tau^{\beta}-1)^{p-2}\,d\tau,
 \end{align*}
 where we have substituted $y =\tau x,\, dy = xd\tau.$

In the last integral we substitute
$$\tau = \frac{1}{t}\quad d\tau = -\frac{dt}{t^2}$$
and obtain
$$\int_{1+\frac{\e}{x}}^{\infty}\!(\tau-1)^{-sp}\tau^{p-1}(\tau^{\beta}-1)^{p-2}\,d\tau\,=\, \int_{0}^{\frac{x}{x+\e}}\!(1-t)^{-sp}t^{\beta-1} (1-t^{\beta})^{p-2}\mathbf{t^{p(s-\beta)}}\,dt.$$

Let us first consider the case $0 < \beta < s$. In that case there is no problem about convergence, and letting $\e \to 0$ some terms cancel so that only the following identity is left:
\begin{align*}
 & \int_{-\infty}^{\infty}\!\frac{|u(y)-u(x)|^{p-2}\bigl(u(y)-u(x)\bigr)}{|y-x|^{1+sp}}\, dy\\& =
  \tfrac{\beta(p-1)}{sp}x^{\beta(p-1)-sp}\int_{0}^1\!(1-t)^{-sp}(1-t^{\beta})^{p-2}t^{\beta-1}\left[t^{p(s-\beta)}-1\right]\,dt
\end{align*}where $x>0.$ It is plain that the integral has the same sign as $\beta-s$, since the factor $[t^{p(s-\beta)}-1]$ determines this. This settles the case $\beta < s.$

If $s=\beta$ most terms cancel immediately, but in this case the integral becomes
$$\int_{1-\frac{\e}{x}}^{1-\frac{\e}{x+\e}}\!(1-t)^{-sp}(1-t^s)^{p-2}t^{s-1}\,dt.$$
It approaches zero as $\e \to 0,$ although the complete integral diverges at the endpoint $t = 1$ if $s >1-\tfrac{1}{p}.$ This proves that $\Lps (x_+)^s = 0$ when $x>0.$
\end{proof}

In order to use $x_+^s$ as a minorant, we modify it near infinity so that it becomes the \emph{bounded} function:
\begin{equation*}
  \ell(x) = \begin{cases}
    0\qquad \,\, x < 0,\\
    x^s\qquad 0<x<L,\\
   L^s \qquad L \leq x  <  \infty.
  \end{cases}
\end{equation*}

\begin{prop}\label{minorant} There is a number $\delta > 0$ such that
  $$-4\,\frac{p-1}{sp}L^{-s}  < \Lps \ell (x) < -\frac{p-1}{sp}L^{-s}\quad \text{
  when}\quad 0<\frac{x}{L} < \delta.$$
\end{prop}

\begin{proof} The only change in the previous calculations  is that the right-hand member
  0 of the equation $\Lps = 0$ should be replaced by twice the negative quantity

  $$-\int_{L}^{\infty}\!\frac{(y^s-x^s)^{p-1}}{(y-x)^{1+sp}}\,dy \,
  +\int_{L}^{\infty}\!\frac{(L^s-x^s)^{p-1}}{(y-x)^{1+sp}}\,dy.
  $$
  Upon a partial integration this correction becomes twice the integral
  \begin{align*}
   &-\tfrac{(p-1)}{p}\int_{L}^{\infty}\!(y-x)^{-sp}y^{s-1}(y^s-x^s)^{p-2}\,dy\\=
    &-\tfrac{(p-1)}{p}\,x^{-s}\int_{\frac{L}{x}}^{\infty}\!(\tau-1)^{-sp}\tau^{s-1}(\tau^s-1)^{p-2}\,d\tau.
  \end{align*}
  When $\frac{x}{L}$ is small enough, the integral is of the magnitude $s^{-1}(\frac{x}{L})^s$ and the correction is of the magnitude $2\tfrac{p-1}{sp}L^{-s}.$ The desired estimate follows.
\end{proof}

 \paragraph{Dead Variables.} If a function of $n$ variables depends only on one of them, say
  $$u(x_1,x_2,\dots,x_n) = u_1(x_1)$$
  then, at least formally,
  \begin{align}\label{reduction}
   &\int\limits_{-\infty}^{\infty}\!\! \int\limits_{-\infty}^{\infty}\!\!\!\cdots \!\!
    \int\limits_{-\infty}^{\infty}\!\frac{|u(y)-u(x)|^{p-2}\bigl(u(y)-u(x)\bigr)}{|y-x|^{n+sp}}\,dy_1\,dy_2\cdots dy_n\nonumber\\=& N(n,sp)\int\limits_{-\infty}^{\infty}\!\frac{|u_1(y_1)-u(x_1)|^{p-2}\bigl(u(y_1)-u(x_1)\bigr)}{|y-x|^{1+sp}}\,dy_1,
  \end{align}
  where the constant $N(n,sp)$ is explicit. To see this, notice first that we may take $x_2 =0, x_3 =0,\dots,x_n=0$ and further observe that
  $$
 \int\limits_{-\infty}^{\infty}\!\! \int\limits_{-\infty}^{\infty}\!\!\!\cdots \!\!\!
 \int\limits_{-\infty}^{\infty}\!\!\frac{dy_2\,dy_3\,\cdots\, dy_n}{\left((y_1\!-\!x_1)^2+ y_2^2+\dots+y_n^2\right)^{\frac{n+sp}{2}}}\, = N(n,sp)|x_1\!-\!y_1|^{-1-sp}.$$
   Then the above $n$-dimensional integral can be written as 
   \begin{equation*} \int\limits_{-\infty}^{\infty}\!\underbrace{|u(y)\!-\!u(x)|^{p-2}\bigl(u(y)\!-\!u(x)\bigr)}_{\text{Depends only on}\,\, x_1,y_1}\biggl\{\int\limits_{-\infty}^{\infty}\!\!\!\!\cdots\!\!\!\!\int \limits_{-\infty}^{\infty}\!\frac{dy_2\,\cdots\, dy_n}{\left((y_1\!-\!x_1)^2+y_2^2+\dots+y_n^2\right)^{\frac{n+sp}{2}}}\biggr\}dy_1,
     \end{equation*}
   which reduces to Equation \eqref{reduction}. Thus, we have obtained the result in $\R^n:$ 
   
 \begin{lemma}\label{half-space} In the half-space $\mathbf{n}\cdot x >0$ 
  $$\Lps ( \mathbf{n}\cdot x )_+^{\beta} = -C(\beta,s,p,n)\,(\mathbf{n}\cdot x)_+^{\beta(p-1)-sp}$$
  where the constant $C(\beta,s,p,n) > 0$ for $0 < \beta < s.$ In the case $\beta = s$ we have
  $$\Lps (\mathbf{n}\cdot x  )_+^s = 0.
  $$
 \end{lemma}

 Also the function $$\ell(x_1,x_2,...,x_2) = \ell(x_1),$$ extended with dead variables, preserves its properties in Proposition \ref{minorant}.

 \begin{cor}\label{half-infimum} The function 
   $$k(x) = (1-|x|)^{\beta}_{+}$$
   is a supersolution in the ball $|x| <1,$
   if $0<\beta \leq s.$ In fact, 
   $$
   \Lps k\leq -C(\beta,s,p,n)<0, \quad \text{ when $\beta<s$}.
   $$
 \end{cor}

 \begin{proof}  The infimum over all the half-plane supersolutions  $\bigl(\mathbf{n}\cdot(\mathbf{n}-x)\bigl)_+^{\beta}$, where $\mathbf{n}$ is the exterior unit normal, is a supersolution in their common domain. (If $\beta \not = s$, one even gets a strict supersolution.)
 \end{proof}
 
 \subsection{A Radial ''Barrier''}

 The proof of Theorem \ref{thm:barrier} about the fundamental barrier function was based on the following result about the radial function 
 \begin{equation*}
   \omega(r) = \begin{cases}
     (r-r_0)^{\beta},\quad r \geq r_0\\
     0,\qquad\quad 0 \leq r \leq r_0
   \end{cases}
 \end{equation*}
 where $r = \sqrt{x_1^2+\dots+x_n^2}$.

 \begin{lemma}\label{ring} Let $0<\beta < s.$ There is a $\delta = \delta(r_0,s,\beta,p,n)>0$ such that
   $$\Lps \omega(r) \leq -1\quad\text{when}\quad r_0 < r  < r_0+\delta.$$
 \end{lemma}

 \begin{proof} The calculations are elementary, but lengthy. We write the details in three dimensions. The operator $\Lps$ preserves rotational symmetry. We bring the point $x$ to the position $(0,0,r)$ and for $y$ we use  spherical coordinates
   \begin{equation*}
     \begin{cases}
       y_1=\rho\sin\theta\cos\phi\\
       y_2 = \rho \sin \theta \sin \phi\\
       y_3 = \rho \cos \theta
     \end{cases}
   \end{equation*}
   so that
   $$|y-x|^2 = (\rho \cos \theta -r)^2 + \rho^2\sin^2\theta \cdot 1= \rho^2 + r^2 -2\rho r \cos \theta.$$
  Writing $\Lps \omega(r)$ in spherical coordinates and integrating away the azimuth $\phi$ we obtain the expression 
   $$2\cdot2\pi\int\limits_0^{\infty}\int\limits_0^{2\pi}\frac{|\omega(\rho)-\omega(r)|^{p-2}\bigl(\omega(\rho)-\omega(r)\bigr)}{\left(\rho^2+r^2-2\rho r\cos \theta \right)^{\frac{3+sp}{2}}}\sin \theta\, \rho^2\,d\theta\, d\rho$$
   for $\Lps \omega(r).$ (Recall that $\Lps$ carries the factor $2$.)
   Using
   \begin{equation*}
     \frac{d\,\,}{d\theta}\,
     \frac{\left(\rho^2+r^2-2\rho r \cos \theta \right)^{-\frac{1+sp}{2}}}{(1+sp)\rho r}
     \,=\, \frac{\sin \theta}{\left(\rho^2+r^2-2\rho r \cos \theta \right)^{\frac{3+sp}{2}}}
   \end{equation*} 
     we  integrate with respect to $\theta$:
     \begin{equation*}  \frac{4\pi}{r(1+sp)}\int\limits_0^{\infty}|\omega(\rho)-\omega(r)|^{p-2}\bigl(\omega(\rho)-\omega(r)\bigr)\!\left[\frac{\rho}{|\rho-r|^{1+sp}}-\frac{\rho}{|\rho+r|^{1+sp}}\right] d\rho
         \end{equation*}
     The integrals
     \begin{align*}
       \int\frac{\rho\,d\rho}{(r-\rho)^{1+sp}} &=\frac{1}{1-sp}(r-\rho)^{1-sp} +\frac{r}{sp}(r-\rho)^{-sp},
\quad  \mathbf{\rho < r}\\     
         \int\frac{\rho\,d\rho}{(\rho-r)^{1+sp}} &=\frac{1}{1-sp}(\rho-r)^{1-sp} -\frac{r}{sp}(\rho-r)^{-sp} , \quad    \mathbf{\rho > r}\\     	     \int\frac{\rho\,d\rho}{(\rho+r)^{1+sp}} &=\frac{1}{1-sp}(\rho+r)^{1-sp} +\frac{r}{sp}(\rho+r)^{-sp},\quad \mathbf{\rho \gtrless r}     \end{align*}
     are convenient later.

     In the calculations below we keep $\mathbf{r > r_0}$
     and $0 < \beta < s.$ First we take the cases $\boxed{sp\not = 1.}$  We split the integral in three pieces.

     \bigskip

\noindent {\bf The part $0<\rho<r_0:$}
     \begin{gather*}     \int\limits_o^{r_0}|0-\omega(r)|^{p-2}(0-\omega(r))\!\left[\frac{\rho}{|\rho-r|^{1+sp}}-\frac{\rho}{|\rho+r|^{1+sp}}\right] d\rho\\
=\, -(r-r_0)^{\beta(p-1)}\left[\frac{1}{1-sp}(r-\rho)^{1-sp}\right. +\frac{r}{sp}(r-\rho)^{-sp}\\\qquad \left.
  - \frac{1}{1-sp}(\rho+r)^{1-sp} -\frac{r}{sp}(\rho+r)^{-sp}\right]_{\rho = r_0-0}
     \end{gather*}
     The substitution at $\rho=0$ yields zero. The value at $\rho = r_0-0$ will cancel against the corresponding quantity below.
     \bigskip

\noindent {\bf The part $r_0<\rho<r:$}
     \begin{multline*}
       \begin{aligned}
       \int\limits_{r_0}^{r}\cdots d\rho &= - \int\limits_{r_0}^{r}\left((r-r_0)^{\beta}-(\rho-r_0)^{\beta}\right)^{p-1}\!\left[\frac{\rho}{|\rho-r|^{1+sp}}-\frac{\rho}{|\rho+r|^{1+sp}}\right] d\rho\\
         &= -\left((r-r_0)^{\beta}-(\rho-r_0)^{\beta}\right)^{p-1}\left[\frac{1}{1-sp}(r-\rho)^{1-sp}\qquad\right.\\&\qquad\left. +\frac{r}{sp}(r-\rho)^{-sp}
           - \frac{1}{1-sp}(\rho+r)^{1-sp} -\frac{r}{sp}(\rho+r)^{-sp}
           \right]_{\rho = r_0+0}^{r-0}
         \end{aligned}\\
             -\tfrac{\beta(p-1)}{sp}\,r \int\limits_{r_0}^{r}\left((r-r_0)^{\beta}-(\rho-r_0)^{\beta}\right)^{p-2}(\rho-r_0)^{\beta-1}\left\{(r-\rho)^{-sp}-(r+\rho)^{-sp}\right\}d\rho\\
               -\tfrac{\beta(p-1)}{1-sp}  \int\limits_{r_0}^{r}\left((r-r_0)^{\beta}-(\rho-r_0)^{\beta}\right)^{p-2}(\rho-r_0)^{\beta-1}     \left\{(r-\rho)^{1-sp}-(r+\rho)^{1-sp}\right\}d\rho
     \end{multline*}
     Now we substitute
     $$\rho-r_0 =t(r-r_0),\qquad d\rho = (r-r_0)\,dt$$
     in the two last integrals:
       \begin{align*} -\tfrac{\beta(p-1)}{sp}\,r(r-r_0)^{\beta(p-1)-sp}\!\int\limits_0^1(1-&t^{\beta})^{p-2}t^{\beta-1}\quad \times\\\Bigl\{&(1-t)^{-sp}
         -\Bigl(1+t+\frac{2r_0}{r-r_0}\Bigr)^{-sp}\Bigr\}dt\\
       -\tfrac{\beta(p-1)}{1-sp}(r-r_0)^{\beta(p-1)-sp+1}\!\int\limits_0^1(1-&t^{\beta})^{p-2}  t^{\beta-1}\quad \times\\ \Bigl\{&(1-t)^{1-sp} 
         -\Bigl(1+t+\frac{2r_0}{r-r_0}\Bigr)^{1-sp}\Bigr\}dt
       \end{align*}

       \bigskip

\noindent {\bf The part $r<\rho<\infty:$}
   \begin{multline*}
       \begin{aligned}
       \int\limits_{r+0}^{\infty}\cdots d\rho &= - \int\limits_{r}^{\infty}\left((\rho-r_0)^{\beta}-(r-r_0)^{\beta}\right)^{p-1}\!\left[\frac{\rho}{(\rho-r)^{1+sp}}-\frac{\rho}{(\rho+r)^{1+sp}}\right] d\rho\\
         &= -\left((\rho-r_0)^{\beta}-(r-r_0)^{\beta}\right)^{p-1}\left[\frac{1}{1-sp}(\rho-r)^{1-sp}\qquad\right.\\&\qquad\left. -\frac{r}{sp}(\rho-r)^{-sp}
           - \frac{1}{1-sp}(\rho+r)^{1-sp} -\frac{r}{sp}(\rho+r)^{-sp}
           \right]_{\rho = r+0}^{\infty}
         \end{aligned}\\
             \tfrac{\beta(p-1)}{sp}\,r \int\limits_{r}^{\infty}\left((\rho-r_0)^{\beta}-(r-r_0)^{\beta}\right)^{p-2}(\rho-r_0)^{\beta-1}\left\{(\rho-r)^{-sp}+(r+\rho)^{-sp}\right\}d\rho\\
               -\tfrac{\beta(p-1)}{1-sp}  \int\limits_{r}^{\infty}\left((\rho-r_0)^{\beta}-(r-r_0)^{\beta}\right)^{p-2}(\rho-r_0)^{\beta-1}     \left\{(\rho-r)^{1-sp}-(r+\rho)^{1-sp}\right\}d\rho
   \end{multline*}
   The substitution term at $\rho =\infty$ vanishes and at $\rho =r+0$ it cancels against a similar term. After this, there are only integrals left.
     Now we substitute
     $$\rho-r_0 =\tau(r-r_0),\quad \rho-r = (\tau-1)(r-r_0),\quad d\rho = (r-r_0)\,d\tau$$
     in the two last integrals:
       \begin{align*} \tfrac{\beta(p-1)}{sp}\,r(r-r_0)^{\beta(p-1)-sp}\!\int\limits_1^{\infty}(\tau^{\beta}-1)^{p-2}&\tau^{\beta-1}\,\, \times\\\Bigl\{&(\tau-1)^{-sp}
         +\Bigl(1+\tau+\frac{2r_0}{r-r_0}\Bigr)^{-sp}\Bigr\}d\tau\\
       -\tfrac{\beta(p-1)}{1-sp}(r-r_0)^{\beta(p-1)-sp+1}\!\int\limits_1^{\infty}(\tau^{\beta}-1)^{p-2} & \tau^{\beta-1}\,\, \times\\ \Bigl\{&(\tau-1)^{1-sp} 
         -\Bigl(1+\tau+\frac{2r_0}{r-r_0}\Bigr)^{1-sp}\Bigr\}d\tau
       \end{align*}
       In order to get command over the sign of the main term, we use the substitution
       $$\tau = \frac{1}{t},\qquad d\tau = - \frac{dt}{t^2}$$
       to bring the intergrations over a common interval. The two last integrals become

         \begin{align*} +\tfrac{\beta(p-1)}{sp}\,r(r-r_0)^{\beta(p-1)-sp}\!\int\limits_0^{1}\mathbf{t^{p(s-\beta)}}&(1-t^{\beta})^{p-2}t^{\beta-1}\quad \times\\\Bigl\{&(1-t)^{-sp}
         +\Bigl(1+t+\frac{2r_0t}{r-r_0}\Bigr)^{-sp}\Bigr\}dt\\
       -\tfrac{\beta(p-1)}{1-sp}(r-r_0)^{\beta(p-1)-sp+1}\!\int\limits_0^{1}\mathbf{t^{p(s-\beta)-1}}&(1-t^{\beta})^{p-2}  t^{\beta-1}\quad \times\\ \Bigl\{&(1-t)^{1-sp} 
         -\Bigl(1+t+\frac{2r_0t}{r-r_0}\Bigr)^{1-sp}\Bigr\}dt.
         \end{align*}

         \emph{In toto,} we obtain  by adding the integrals left in the three cases
         \begin{equation}\label{fyra}\displaystyle
           \Lps \omega(r) =\tfrac{4\pi}{1+sp} \tfrac{\beta(p-1)}{sp}\,r(r-r_0)^{\beta(p-1)-sp}\times (\mathrm{I}+\mathrm{I\!I}+\mathrm{I\!I\!I}+\mathrm{I\!V}),\end{equation}
         where  four integrals appear. The sign of the whole expression is determined by the factor $\mathrm{I}+\mathrm{I\!I}+\mathrm{I\!I\!I}+\mathrm{I\!V}$. First,  we have the \textsf{main term:}
         \begin{equation*}\label{romerska}\displaystyle
           \mathrm{I} = \int\limits_{0}^{1}\left(t^{p(s-\beta)}-1\right)(1-t^{\beta})^{p-2}t^{\beta-1}(1-t)^{-sp}\,dt.
         \end{equation*}
         It is of utmost importance that this integral is \emph{strictly negative} when $0 < \beta < s.$ \emph{It will dominate over the other terms} provided that $r-r_0$ is small. The next one is
         \begin{equation*}\displaystyle
           \mathrm{I\!I} = \int\limits_{0}^{1}(1-t^{\beta})^{p-2}t^{\beta-1}\left(1+t^{p(s-\beta)}\right)\underbrace{\Bigl(1+t+\frac{2r_0}{r-r_0}\Bigr)^{-sp}}_{\text{Goes to}\, 0 \,\text{as}\, r\to r_0+0}\, dt. \end{equation*}
         Notice that this integral is uniformly small when $0<r-r_0<\delta =$ a small number. The third term is
         \begin{equation*}
           \mathrm{I\!I\!I} = \tfrac{sp}{1-sp}\frac{(r-r_0)}{r}\int\limits_0^1 (1-t^{\beta})^{p-2}t^{\beta-1}\Bigl\{\Bigl(1+t+\frac{2r_0}{r-r_0}\Bigr)^{1-sp} - (1-t)^{1-sp} \Bigr\}\,dt.
         \end{equation*}
          The factor $(r-r_0)$ in front of the term $\mathrm{I\!I\!I}$  makes it small compared to the main term. Finally,
         \begin{align*}\displaystyle
           \mathrm{I\!V} = \tfrac{sp}{1-sp}\frac{(r-r_0)}{r} \int\limits_{0}^{1}t^{p(s-\beta)-1}&(1-t^{\beta})^{p-2}t^{\beta-1} \,\, \times \\&\Bigl\{ \Bigl(1+t+\frac{2r_0t}{r-r_0}\Bigr)^{1-sp} -(1-t)^{1-sp}  \Bigr\}\, dt. \end{align*}
         This term is problematic, because the integral would not converge at $t = 0$ without the expression in braces. We can write\footnote{The case $sp \leq  1$ is urgent, but the estimation also works for $sp > 1.$}
         \begin{align*}
          &\tfrac{1}{1-sp} \Bigl\{\Bigl(1+t+\frac{2r_0t}{r-r_0}\Bigr)^{1-sp} -(1-t)^{1-sp}\Bigr\} \\
          & =   (1+\xi)^{-sp}\left[2t + \frac{2r_0t}{r-r_0}\right]\\
          & < 2(1-t)^{-sp}\left[t + \frac{r_0t}{r-r_0}\right],\\
           \end{align*}
           where the Mean Value Theorem was used for the function $(1+y)^{1-sp}$ and $-t<\xi < t +\tfrac{2r_0t}{r-r_0}.$ Thus we have the estimate           
$$
 (r-r_0)\Bigl\vert \Bigl(1+t+\frac{2r_0t}{r-r_0}\Bigr)^{1-sp} -(1-t)^{1-sp}  \Bigr\vert 
   \leq  2|1-sp| (1-t)^{-sp} r\mathbf{t},
  $$
   where we have gained one power of $t$. Split the integral in $\mathrm{I}\!\mathrm{V}$ as
   $$(r-r_0)\left(\int_0^{\sigma}+\int_{\sigma}^1\right),$$
   where the above estimate shows that we now can make the quantity
   $$(r-r_0)\int_0^{\sigma}\cdots\, dt$$
   as small as we wish, by adjusting  $\sigma$. The  remaining integral over $[\sigma,1]$ multiplied by $r-r_0$ approaches zero, as $r\to r_0$.

   In conclusion, the main term dominates in some small interval $r_0 < r < r_0 + \delta.$ This was the case $sp\not = 1.$

   The case $\boxed{sp=1}$	 requires minor modifications. Doing the same calculations again, the final formula (\ref{fyra}) has again four integrals. The main term $\mathrm{I}$ is the same. So is
   $\mathrm{I\!I}.$ But in $\mathrm{I\!I\!I}$ and $\mathrm{I\!V}$ one has to change the expressions of the type
   $$\frac{1}{1-sp}\Bigl\{\ldots\Bigr\}$$
   using the limit procedure
   $$\lim_{sp\to 1}\frac{b^{1-sp}-a^{1-sp}}{1-sp} \,=\, \ln\left(\frac{b}{a}\right).$$
   After these replacements, the proof goes as above.
 \end{proof}

 \paragraph{One dimension.} The function
 $\omega_1(x) = (|x|-1)_+^{\beta} $ of one variable  satisfies the equation below, when $|x|>1$. We omit the calculations, which are of the same kind as in the three dimensional situation, although shorter. We remark that the function $(|x|^2-1)_+^{\beta}$ seems to be nicer, but it offers us no advantages in the calculations, on the contrary it produces longer expressions.
 \begin{align*}\Lps \omega_1(x) =
   \frac{\beta(p-1)}{sp}& (|x|-1)^{\beta(p-1)-sp}\\ \times\quad \Biggl\{ &\int\limits_0^1\!(t^{p(s-\beta)}-1)(1-t^{\beta})^{p-2}(1-t)^{-sp}t^{\beta-1}\, dt\\ +&\int\limits_0^1\!\Bigl[1+t+\frac{2}{|x|-1}\Bigr]^{-sp}(1-t^{\beta})^{p-2}t^{\beta-1}
   \, dt\\
+ &\int\limits_0^1\! \Bigl[1+t+\frac{2t}{|x|-1}\Bigr]^{-sp}t^{p(s-\beta)}  (1-t^{\beta})^{p-2}t^{\beta-1}\,dt \Biggr\}
 \end{align*}
 The first integral is \emph{strictly negative} when $0 < \beta < s$ and the two other integrals approach zero as $x \to 1+0$ or $x \to -1-0$. Again we get \emph{a strict supersolution} in a thin ''shell'' $1 < |x| < 1 + \delta.$\\
 
\bibliographystyle{amsrefs}
\bibliography{ref}
\end{document}